\documentclass[11pt]{article}
\usepackage{color,amsthm,amsmath,amsfonts,xspace,caption, graphicx,cite}
\usepackage{MnSymbol}
\usepackage{hyperref}
\usepackage{authblk}
\usepackage{dsfont}
\usepackage[top=1in, bottom=1in, left=1in, right=1in]{geometry}

\newcommand{\Poi}{\mathrm{Poi}}
\newcommand{\indic}{\mathds{1}}
\newcommand{\Var}{\mathrm{Var}}
\newcommand{\EE}{\mathbb{E}}
\newcommand{\PP}{\mathbb{P}}

\DeclareMathOperator*{\Exp}{\EE}
\newcommand{\Paren}[1]{\left(#1\right)}
\newcommand{\Abs}[1]{\left|#1\right|}
\newcommand{\Br}[1]{\left[#1\right]}
\newcommand{\Brace}[1]{\left\{#1\right\}}
\newtheorem{Theorem}{Theorem}
\newtheorem{Corollary}{Corollary}
\newtheorem{Lemma}{Lemma}
\newtheorem{definition}{Definition}

\DeclareMathOperator{\sign}{sign}

\title{
Data Amplification: Instance-Optimal Property Estimation 
}
\author{Yi Hao}
\author{Alon Orlitsky}
\affil{\{yih179, alon\} @eng.ucsd.edu \\University of California, San Diego}
\begin{document}
\begin{titlepage}
\clearpage\maketitle
\thispagestyle{empty}
The best-known and most commonly used distribution-property estimation
technique uses a plug-in estimator, with empirical frequency replacing
the underlying distribution. We present novel linear-time-computable
estimators that significantly ``amplify'' the effective amount of data available. 
For a large variety of distribution properties including four of the most popular ones and for every
underlying distribution, 
they achieve the accuracy that the empirical-frequency plug-in
estimators would attain using a logarithmic-factor more samples.

Specifically, for Shannon entropy and a very broad class of properties including $\ell_1$-distance, the new estimators use $n$ samples to achieve the accuracy attained
by the empirical estimators with $n\log n$ samples. For support-size
and coverage, the new estimators use $n$ samples to achieve the
performance of empirical frequency with sample size
$n$ times the logarithm of the property value.
Significantly strengthening the traditional min-max formulation, these results hold
not only for the worst distributions, but for each and every underlying
distribution. Furthermore, the logarithmic amplification factors are optimal.
Experiments on a wide variety of distributions show
that the new estimators outperform the previous state-of-the-art estimators
designed for each specific property. 
\end{titlepage}
\section{Introduction}

Recent years have seen significant interest in estimating properties of
discrete distributions over large domains~\cite{ventro,mmentro,mmcover,pnas,mmsize,jvhw}. Chief among these properties
are support size and coverage, Shannon entropy, and $\ell_1$-distance to a known distribution. 
The main achievement of these papers is essentially estimating several 
properties of distributions with alphabet size $k$ using just $k/\log k$
samples. 

In practice however, the underlying distributions are often simple,
and their properties can be accurately estimated with significantly fewer than $k/\log k$ samples. For example, if the distribution is concentrated on a
small part of the domain, or is exponential, very few samples
may suffice to estimate the property. To address this discrepancy,
~\cite{nips} took the following competitive approach. 

The best-known distribution property estimator
is the \emph{empirical estimator} that replaces the unknown underlying distribution
by the observed empirical distribution. For example, with $n$ samples,
it estimates entropy by $\sum_i (N_i/n)\log(n/N_i)$ where $N_i$
is the number of times symbol $i$ appeared. Besides its simple and
intuitive form, the empirical estimator is also consistent, stable,
and universal. It is therefore the most commonly used property
estimator for data-science applications. 

The estimator derived in~\cite{nips} uses $n$ samples and for
any underlying distribution achieves the same performance that
the empirical estimator would achieve with $n\sqrt{\log n}$ samples. 
It therefore provides an effective way to \emph{amplify} the amount of data
available by a factor of $\sqrt{\log n}$, regardless of
the domain or structure of the underlying  distribution.

In this paper we present novel estimators that increase the
amplification factor for all sufficiently smooth properties including those mentioned above from $\sqrt{\log n}$
to the information-theoretic bound of $\log n$.
Namely, for \emph{every} distribution their expected estimation error
with $n$ samples is that of the empirical estimator with $n\log n$
samples and no further uniform amplification is possible. 

It can further be shown~\cite{ventro,mmentro,mmcover,jvhw} that
the empirical estimator estimates all of the above four properties
with linearly many samples, hence the
sample size required by the new estimators is always at most
the $k/\log k$ guaranteed by the state-of-the-art estimators.

The current formulation has several additional advantages
over previous approaches. 
\vspace*{-0.9em}
\paragraph{Fewer assumptions}
It eliminates the need for some commonly used assumptions.
For example, support size cannot be estimated with any number of
samples, as arbitrarily-many low-probabilities may be missed.
Hence previous research~\cite{mmsize,mmcover} unrealistically assumed prior knowledge of
the alphabet size $k$, and additionally that all positive
probabilities exceed $1/k$. 
By contrast, the current formulation does not need these assumptions.
Intuitively, if a symbol's probability is so small that it won't be detected
even with $n\log n$ samples, we do not need to worry about it.
\vspace*{-0.9em}\paragraph{Refined bounds}
For some properties, our results are more refined
than previously shown. For example, existing results estimate
the support size to within $\pm\epsilon k$, rendering the
estimates rather inaccurate when the true support size $S$
is much smaller than $k$. 
By contrast, the new estimation errors are bounded by $\pm\epsilon S$,
and are therefore accurate regardless of the support size.
A similar improvement holds for support coverage.
\vspace*{-0.9em}
\paragraph{Graceful degradation}
For the previous results to work, one needs at
least $k/\log k$ samples. With fewer samples, the estimators
have no guarantees. By contrast, the guarantees of the new
estimators work for any sample size $n$. If $n<k/\log k$, the performance
may degrade, but will still track that of the empirical estimators
with a factor $\log n$ more samples.
\vspace*{-0.9em}
\paragraph{Instance optimality}
With the recent exception of~\cite{nips}, all 
modern property-estimation research took a min-max-related
approach, evaluating the estimation improvement based on the worst possible
distribution for the property. In reality, practical distributions 
are rarely the worst possible and often quite simple, rendering min-max approach
overly pessimistic, and its estimators, typically suboptimal in
practice. In fact, for this very reason, practical distribution
estimators do not use min-max based approaches~\cite{gs95}.
By contrast, our \emph{competitive}, or \emph{instance-optimal},
approach provably ensures amplification for every underlying distribution,
regardless of its complexity.

In addition, the proposed estimators run in time linear in the
sample size, and the constants involved are very small, properties
shared by some, though not all existing estimators. 

We formalize the foregoing discussion in the following definitions.

Let $\Delta_k$ denote the collection of discrete distributions over
$[k]:=\{1,\ldots, k\}$. A distribution \emph{property} is a mapping
$F:\Delta_k\to \mathbb{R}$.
It is \emph{additive} if it can be written as
\[
F(\vec{p}):=\sum_{i\in[k]} f_i(p_i),
\]
where $f_i:[0,1] \to \mathbb{R}$ are real functions. Many important distribution properties are additive:
\paragraph{Shannon entropy} $H(\vec{p}):=\sum_{i\in[k]}-p_i\log p_i$,
is the principal measure of information~\cite{info}, and arises in a
variety of machine-learning~\cite{chowliu,qkc,bg},
neuroscience~\cite{snm, mzfstj, vrlskb}, and other applications.
\paragraph{$\boldsymbol{\ell_1}$-distance}
$D_{\vec{q}}(\vec{p}):=\sum_{i\in[k]}|p_i-q_i|$, where $\vec{q}$ is a given distribution, is one of the most basic
and well-studied properties in the field of
distribution property testing~\cite{bfrsw, bffkrw, rd10, testingu}.
\paragraph{Support size} $S(\vec{p}):=\sum_{i\in[k]}\indic_{p_i>0}$, is a fundamental quantity for discrete distributions, and plays an important role in vocabulary size~\cite{vocabulary, te87, et76} and population estimation~\cite{ mcxlbg, population}.
\paragraph{Support coverage}
$C(\vec{p}):=\sum_{i\in[k]}(1-(1-p_i)^{m})$, for a given $m$, 
represents the number of distinct elements we would expect to see in $m$
independent samples, arises in many
ecological~\cite{ecological, ca84, ca92, ca14}, biological~\cite{klr99, ccover}, genomic~\cite{genomic} as well as database~\cite{database} studies. \\

Given an additive property $F$ and sample access to an
unknown distribution $\vec{p}$, we would like to estimate the value of
$F(\vec{p})$ as accurately as possible. Let $[k]^n$ denote
the collection of all length-$n$ sequences, an estimator is a function
$\hat{F}:[k]^n\to \mathbb{R}$ that maps a sample sequence $X^n\sim
\vec{p}$ to a property estimate $\hat{F}(X^n)$.  We evaluate the
performance of $\hat{F}$ in estimating $F(\vec{p})$ via its
\emph{mean absolute error} (MAE),
\[
L(\hat{F},\vec{p}, n):=\Exp_{X^n\sim \vec{p}}\Abs{\hat{F}(X^n)-F(\vec{p})}.
\]
Since we do not know $\vec{p}$, the common approach is to consider the worst-case MAE of $\hat{F}$ over $\Delta_k$,
\[
L(\hat{F}, n):=\max_{\vec{p}\in \Delta_k}L(\hat{F},\vec{p}, n).
\]
The best-known and most commonly-used property estimator
is the \emph{empirical plug-in estimator}.
Upon observing $X^n$, let $N_i$ denote the number of times symbol
$i\in[k]$ appears in $X^n$.
The empirical estimator estimates $F(\vec{p})$ by
\[
\hat{F}^E(X^n):=\sum_{i\in[k]}f_i\Paren{\frac{N_i}{n}}.
\]
Starting with Shannon entropy, it has been
shown~\cite{mmentro} that for $n\ge k$, the worst-case MAE of
the empirical estimator $\hat{H}^E$ is
\begin{equation}
\label{eqn:emp_loss}
L(\hat{H}^E, n)=\Theta\Paren{{\frac{k}{n}}+\frac{\log k}{\sqrt{n}}}.
\end{equation}
On the other hand,~\cite{ventro,mmentro,mmcover,jvhw} showed that for
$n\ge k/{\log k}$, more sophisticated estimators
achieve the best min-max performance of \vspace{-0.3em}
\begin{equation}
\label{eqn:mm_loss}
L(n)
:=
\min_{\hat F}\max_{\vec{p}\in \Delta_k}L(\hat{F},\vec{p}, n)
=
\Theta\Paren{\frac{k}{n\log n}+\frac{\log k}{\sqrt{n}}}.
\end{equation}
Hence up to constant factors, for the ``worst'' distributions,
the  MAE of these estimators with $n$ samples equals that of the
empirical estimator with $n\log n$ samples.
A similar relation holds for the other three properties we consider.

However, the min-max formulation is pessimistic as it evaluates the estimator's
performance based on its MAE for the worst distributions. 
In many practical applications, the underlying distribution is fairly
simple and does not attain this worst-case loss, rather, a much
smaller MAE can be achieved. Several
recent works have therefore gone beyond worst-case analysis and
designed algorithms that perform well for all distributions,
not just those with the worst performance~\cite{compdist, instdist}.

For property estimation,~\cite{nips} designed an estimator $\hat{F}^A$
that for any underlying distribution uses $n$ samples to achieve the
performance of the $n\sqrt{\log n}$-sample empirical estimator, 
hence effectively multiplying the data size by a $\sqrt{\log n}$
\emph{amplification factor}.

\begin{Lemma}~\cite{nips}
For every property $F$ in a large class that includes the four
properties above, there is an absolute constant $c_F$ such that
for all distribution $\vec{p}$ and all $\varepsilon\leq 1$,
\[
L(\hat{F}^A,\vec{p}, n)
\le
L\Paren{\hat{F}^E,\vec{p}, \varepsilon n\sqrt{\log n}} + c_F\cdot \varepsilon.
\]
\end{Lemma}

In this work, we fully strengthen the above result and establish 
the limits of data amplification for all sufficiently smooth 
additive properties including
four of the most important ones. 
Using Shannon entropy as an example, 
we achieve a $\log n$ amplification factor.
Equations~\eqref{eqn:emp_loss} and~\eqref{eqn:mm_loss} imply that 
the improvement over the empirical estimator cannot always 
exceed $\mathcal{O}(\log n)$, 
hence up to a constant, this amplification factor is
information-theoretically optimal. Similar optimality arguments hold for 
our results on the other three properties. 

Specifically, we derive linear-time-computable estimators
$\hat{H}$, $\hat{D}$, $\hat{S}$, $\hat{C}$, and $\hat{F}$ for Shannon entropy, $\ell_1$-distance, 
support size, support coverage, and a broad class of additive properties which we refer to as ``Lipschitz properties''.
These estimators take a single parameter $\varepsilon$, and given
samples  $X^n$, amplify the data as described below.

Let $a\land b:=\min\{a, b\}$ and abbreviate the support size $S({\vec{p}})$ by $S_{\vec{p}}$. For some absolute constant $c$, 
the following five theorems hold for all $\varepsilon\leq 1$, 
all distributions $\vec{p}$, and all $n\ge1$.
\begin{Theorem}[Shannon entropy]\label{thm1} 
\[
L(\hat{H},\vec{p}, n)\le L\Paren{\hat{H}^E,\vec{p}, {\varepsilon n\log n}}
+c\cdot  \Paren{\varepsilon \land \Paren{\frac{S_{\vec{p}}}{n}+\frac{1}{n^{0.49}}}}.
\]
\end{Theorem}
Note that the estimator does not need to know $S_{\vec{p}}$ or $k$. When $\varepsilon=1$, the estimator amplifies the data by a factor 
of $\log n$. As $\varepsilon$ decreases, the amplification factor decreases,
and so does the extra additive inaccuracy. One can also set $\varepsilon$ to be 
a vanishing function of $n$, e.g., $\varepsilon=1/\log\log n$. This result may be interpreted
 as follows. For distributions with large support sizes such that the min-max estimators provide no or only very weak guarantees, our estimator with $n$ samples always tracks the performance of the $n\log n$-sample empirical estimator. On the other hand, for distributions with relatively small support sizes, our estimator achieves a near-optimal $\mathcal{O}(S_{\vec{p}}/ n)$-error rate.
 
 In addition, the above result together with Proposition 1 in~\cite{pentro} trivially implies that 
 \begin{Corollary}
 In the large alphabet regime where $n=o(k/\log k)$, the min-max MAE of estimating Shannon entropy satisfies 
 \[
 L(n)\leq(1+o(1))\log\Paren{1+\frac{k-1}{n\log n}}.
 \]  
 \end{Corollary}
 \pagebreak
 
Similarly, for $\ell_1$-distance,

\begin{Theorem}[$\ell_1$-distance]\label{thm2.1} 
For any $\vec{q}$, we can construct an estimator $\hat{D}$ for $D_{\vec{q}}$ such that
\[
L(\hat{D},\vec{p}, n)
\le
L\Paren{\hat{D}^E,\vec{p}, \varepsilon^2 {n\log n}}+c\cdot \Paren{\varepsilon\land \Paren{\sqrt{\frac{S_{\vec{p}}}{n}}+\frac{1}{n^{0.49}}}}. 
\]
\end{Theorem} 
Besides having an interpretation similar to Theorem~\ref{thm1}, the above result shows that for each $\vec{q}$ and each $\vec{p}$, we can use just $n$ samples to achieve the performance of the $n\log n$-sample empirical estimator.
More generally, for any additive property $F(\vec{p})=\sum_{i\in[k]} f_i(p_i)$ that satisfies the simple condition: $f_i$ is $\mathcal{O}(1)$-Lipschitz, for all $i$,
\begin{Theorem}[General additive properties]\label{thm2.2} 
Given $F$, we can construct an estimator $\hat{F}$ such that
\[
L(\hat{F},\vec{p}, n)
\le
L\Paren{\hat{F}^E,\vec{p}, \varepsilon^2 {n\log n}}+\mathcal{O}\Paren{\varepsilon\land \Paren{\sqrt{\frac{S_{\vec{p}}}{n}}+\frac{1}{n^{0.49}}}}. 
\]
\end{Theorem} 
We refer to the above general distribution property class as the class of ``Lipschitz properties''. Note that the $\ell_1$-distance $D_{\vec{q}}$, for any $\vec{q}$, clearly belongs to this class. 

Lipschitz properties are essentially bounded by absolute constants and Shannon entropy
grows at most logarithmically in the support size, and we were able
to approximate all with just an additive error. Support size
and support coverage can grow linearly in $k$ and $m$, and can be approximated
only multiplicatively. We therefore evaluate the estimator's
normalized performance. 

Note that for both properties, the amplification factor is 
logarithmic in the property value, which can be arbitrarily larger
than the sample size $n$. The following two theorems hold for $\epsilon\leq e^{-2}$,

\begin{Theorem}[Support size]\label{thm3}
\[
\frac{1}{S_{\vec{p}}}L(\hat{S},\vec{p}, n) \leq \frac{1}{S_{\vec{p}}}L\Paren{\hat{S}^E,\vec{p}, |\log^{-2}{\varepsilon}|\cdot n\log S_{\vec{p}} } +{c}
\Paren{S_{\vec{p}}^{|\log^{-1}{\varepsilon}|-\frac12}+\varepsilon}.
\]
\end{Theorem}
To make the slack term vanish, one can simply set $\varepsilon$ to be a vanishing function of $n$ (or $S_{\vec{p}}$), e.g., $\varepsilon=1/\log n$. Note that in this case, the slack term modifies the multiplicative error in estimating $S_{\vec{p}}$ by only $o(1)$, which is negligible in most applications. 
Similarly, for support coverage, 
\begin{Theorem}[Support coverage]\label{thm4}
Abbreviating $C({\vec{p}})$ by $C_{\vec{p}}$,
\[
\frac{1}{C_{\vec{p}}}L(\hat{C},\vec{p}, n) \leq \frac{1}{C_{\vec{p}}}L\Paren{\hat{C}^E,\vec{p}, |\log^{-2}{\varepsilon}|\cdot n\log C_{\vec{p}}} +c\Paren{C_{\vec{p}}^{|\log^{-1}{\varepsilon}|-\frac12}+\varepsilon}.
\]
\end{Theorem} 
For notational convenience, let $h(p):=-p\log p$ for entropy, $\ell_{q}(p):=|p-q|-q$ for $\ell_1$-distance, $s(p):=\indic_{p>0}$ for support size, and $c(p):=1-(1-p)^{m}$ for support coverage. 
In the next section, we provide an outline of the remaining contents, and a high-level overview of our techniques.

\section{Outline and technique overview}
In the main paper, we focus on Shannon entropy and prove a
weaker version of 
Theorem~\ref{thm1}.
\begin{Theorem}\label{thm5}
For all $\varepsilon\leq 1$ and all distributions $\vec{p}$, the estimator $\hat{H}$ described in Section~\ref{Est_const} satisfies
\[
L(\hat{H},\vec{p}, n)\le L(\hat{H}^E,\vec{p}, \varepsilon n\log n)+\Paren{1+c\cdot \varepsilon}\land \Paren{\frac{S_{\vec{p}}}{\varepsilon n}+\frac{1}{n^{0.49}}}.
\]
\end{Theorem}
The proof of Theorem~\ref{thm5} in the rest of the paper is organized as follows. 
In Section~\ref{useful_tools}, we present a few useful concentration inequalities for Poisson and binomial random variables.
In Section~\ref{Bernstein}, we relate the bias of
the $n$-sample empirical estimator to the degree-$n$ Bernstein
polynomial $B_n(h,x)$ by $B_n(h,p_i)=\EE[h(N_i/n)]$.
In Section~\ref{der_Bern}, we show that the
absolute difference between the \emph{derivative} of $B_n(h,x)$ and a simple
function $h_n(x)$ is at most $1$, uniformly for all $x\leq 1-(n-1)^{-1}$. 

Let $a:=\varepsilon\log n$ be the amplification parameter. 
In Section~\ref{approx_der} we approximate $h_{na}(x)$
by a degree-$\Theta(\log n)$ polynomial $\tilde{h}_{na}(x)$,
 and bound the approximation error uniformly by $c\cdot \varepsilon$.
Let $\tilde{H}_{na}(x):= \int_0^{x}\tilde{h}_{na}(t) dt$.
By construction, $|B_{na}'(h,x)-\tilde{h}_{na}(x)|\leq |B_{na}'(h,x)-{h}_{na}(x)|+|{h}_{na}(x)-\tilde{h}_{na}(x)|\leq 1+c\cdot \varepsilon$, implying
$|\tilde{H}_{na}(x)-B_{na}(h,x)|\le x(1+c\cdot \varepsilon)$.

In Section~\ref{Est_const}, we construct our
estimator $\hat{H}$ as follows. First, we divide the symbols into
small- and large- probability symbols according to their counts
in an independent $n$-element sample sequence. 
The concentration inequalities in Section~\ref{useful_tools} imply that this step can be performed with relatively high confidence.
Then, we estimate the partial entropy of each
small-probability symbol $i$ with a near-unbiased estimator of
$\tilde{H}_{na}(p_i)$, and the combined partial entropy of the large-probability symbols
with a simple variant of the empirical estimator. The final estimator
is the sum of these small- and large- probability estimators.

In Section~\ref{Bias_bound}, we bound the bias of $\hat{H}$.
In Sections~\ref{Bias_small} and~\ref{Bias_large}, we 
use properties of $\tilde{H}_{na}$ and the Bernstein polynomials
to bound the partial biases of the small- and large-probability estimators in
terms of $n$, respectively. The key observation is $|\sum_{i}\tilde{H}_{na}(p_i)-B_{na}(h,p_i)|\le \sum_{i}p_i(1+c\cdot \varepsilon)=1+c\cdot \varepsilon$, 
implying that the small-probability estimator has a small bias. To bound the bias of the large-probability estimator, 
we essentially rely on the elegant inequality $|B_n(h,x)-h(x)|\leq 1/n$.

By the triangle inequality, it remains
to bound the mean absolute deviation of $\hat{H}$. We bound this
quantity by bounding the
partial variances of the small- and large- probability estimators in
Section~\ref{Var_small} and Section~\ref{Var_large}, respectively. 
Intuitively speaking, the small-probability estimator has a small variance because it is constructed based on a low-degree polynomial;
  the large-probability estimator has a small variance because $h(x)$ is smoother for larger values of $x$. 

To demonstrate the efficacy of our methods, in Section~\ref{experiments},
we compare the experimental performance of our estimators with that of
the state-of-the-art property estimators for Shannon entropy and
support size over nine distributions. Our competitive estimators outperformed these existing algorithms on nearly all the experimented instances.

Replacing the simple function $h_{n}(x)$ by a much finer approximation of $B_n(h,x)$, we establish the full version of Theorem~\ref{thm1} in Appendix~\ref{refined_est}. 
Applying similar techniques, we prove the other four results in Appendices~\ref{comp_uniformity} (Theorem~\ref{thm2.1} and~\ref{thm2.2}),~\ref{comp_size} (Theorem~\ref{thm3}),
and~\ref{comp_coverage} (Theorem~\ref{thm4}).

\section{Concentration inequalities}\label{useful_tools}
The following lemma gives tight tail probability bounds for Poisson and binomial random variables.
\begin{Lemma}\label{tailprob}~\cite{concen}
Let $X$ be a Poisson or binomial random variable with mean $\mu$, then for any $\delta>0$, 
\[
\PP(X\geq{(1+\delta)\mu})\leq{{\left(\frac{e^\delta}{(1+\delta)^{(1+\delta)}}\right)}^{\mu}}\leq{e^{-(\delta^2\land\delta)\mu/3}},
\]
and for any $\delta\in{(0, 1)}$, 
\[
\PP(X\leq{(1-\delta)\mu})\leq{{\left(\frac{e^{-\delta}}{(1-\delta)^{(1-\delta)}}\right)}^{\mu}}\leq{e^{-\delta^2\mu/2}}.
\]
\end{Lemma}

\pagebreak
\section{Approximating Bernstein polynomials}\label{Bernstein}
With $n$ samples, the bias of the empirical estimator in estimating $H(\vec{p})$ is
\[
\text{Bias}(\hat{H}^E, n):= \EE[\hat{H}^E(X^n)]-H(\vec{p}).
\]
By the linearity of expectation, the right-hand side equals
\[
\EE[\hat{H}^E(X^n)]-H(\vec{p}):=\sum_{i\in[k]}\Paren{\EE\left[h\Paren{\frac{N_i}{n}}\right]-h(p_i)}.
\]
Noting that the degree-$n$ Bernstein polynomial of $h$ is
\[
B_n(h,x):= \sum_{j=0}^{n}h\Paren{\frac{j}{n}}\binom{n}{j}x^j(1-x)^{n-j},
\]
we can express the bias of the empirical estimator as
\[
\text{Bias}(\hat{H}^E, n) = \sum_{i\in[k]}\Paren{B_n(h,p_i)-h(p_i)}.
\]
Given a sampling number $n$ and a parameter $\varepsilon\leq1$, define the
amplification factor $a:=\varepsilon \log n$.
Let $c_l$ and $c_s$ be sufficiently large and small constants, respectively. 
In the following sections, we find a polynomial $\tilde{h}_{na}(x)$ of degree $d-1:=c_s\log n-1$, whose error in approximating 
$B_{na}'(h, x)$ over $I_n:=[0,c_l\log n/n]$ satisfies
\[
|B_{na}'(h,x)-\tilde{h}_{na}(x)|\leq {1+\mathcal{O}\Paren{\varepsilon}}.
\]
Through a simple argument, the degree-$d$ polynomial
\[
\tilde{H}_{na}(x):= \int_0^{x}\tilde{h}_{na}(t) dt,
\]
approximates $B_{na}(h, x)$ with the following pointwise error guarantee. 
\begin{Lemma}
For any $x\in I_n$,
\[
|B_{na}(h,x)-\tilde{H}_{na}(x)|\leq x\Paren{1+\mathcal{O}\Paren{\varepsilon}}.
\]
\end{Lemma}
In Section~\ref{der_Bern}, we relate $B_n'(h,x)$ to a simple function $h_n(x)$, which can be expressed in terms of $h(x)$. In Section~\ref{approx_der}, we approximate $h_n(x)$ by a linear combination of degree-$d$ min-max polynomials of $h(x)$ over different intervals. The resulting polynomial is $\tilde{h}_{na}(x)$.
\subsection{The derivative of a Bernstein polynomial}\label{der_Bern}
According to~\cite{ber}, the first-order derivative of the Bernstein polynomial $B_n(h,x)$ is
\[
B_n'(h,x):= \sum_{j=0}^{n-1}n\Paren{h\Paren{\frac{j+1}{n}}-h\Paren{\frac{j}{n}}}\binom{n-1}{j}x^{j}(1-x)^{(n-1)-j}.
\]
Letting
\[
h_n(x):=n\Paren{h\Paren{\Paren{\frac{n-1}{n}}x+\frac{1}{n}}-h\Paren{\Paren{\frac{n-1}{n}}x}},
\]
we can write $B_n'$ as
\[
B_n'(h,x) = \sum_{j=0}^{n-1}h_n\Paren{\frac{j}{n-1}}\binom{n-1}{j}x^{j}(1-x)^{(n-1)-j}=B_{n-1}(h_n,x).
\]
Recall that $h(x)=-x\log x$. After some algebra, we get
\[
h_n(x) = -\log\frac{n-1}{n}+(n-1)\Paren{h\Paren{x+\frac{1}{n-1}}-h(x)}.
\]
Furthermore, using properties of $h(x)$~\cite{entro}, we can bound the absolute difference between $h(x)$ and its Bernstein polynomial as follows.
\begin{Lemma}\label{Bern_error}
For any $m>0$ and $x\in[0,1]$,
\[
-\frac{1-x}{m} \leq B_{m}(h,x)-h(x) \leq 0.
\]
\end{Lemma}
As an immediate corollary, 
\begin{Corollary}\label{Bern_error_h_n}
For $x\in [0, 1-(n-1)^{-1}]$,
\[
|B_n'(h,x)-h_n(x)|=|B_{n-1}(h_n,x)-h_n(x)|\leq 1.
\]
\end{Corollary}
\begin{proof}
By the equality $B_n'(h,x)=B_{n-1}(h_n,x)$ and Lemma~\ref{Bern_error}, for $x\in [0, 1-(n-1)^{-1}]$,
\begin{align*}
|B_{n-1}(h_n,x)-h_n(x)|
& \leq (n-1)|(B_{n-1}(h,x+(n-1)^{-1})-h(x+(n-1)^{-1}))\\&-(B_{n-1}(h,x)-h(x))|\\
& \leq (n-1)\left|\max\left\{\frac{1-x-(n-1)^{-1}}{n-1},\frac{1-x}{n-1}\right\}\right|\\
& \leq 1. 
\end{align*}
\end{proof}

\subsection{Approximating the derivative function}\label{approx_der}
Denote the degree-$d$ min-max polynomial of $h$ over $[0, 1]$ by
\[
\tilde{h}(x):= \sum_{j=0}^{d} b_j x^j.
\]
As shown in~\cite{mmentro}, the coefficients of $\tilde{h}(x)$ satisfy
\[
|b_j|\leq \mathcal{O}(2^{3d}),
\]
and the error of $\tilde{h}(x)$ in approximating $h(x)$ are bounded as
\[
\max_{x\in [0,1]}|h(x)-\tilde{h}(x)|\leq \mathcal{O}\Paren{\frac{1}{\log^2 n}}.
\]
By a change of variables, the degree-$d$ min-max polynomial of $h$ over $I_n:=[0,c_l\log n/n]$ is 
\[
\tilde{h}_1(x):= \sum_{j=0}^{d} b_j \Paren{\frac{n}{c_l\log n}}^{j-1}x^j+\Paren{\log\frac{n}{c_l\log n}}x.
\]
Correspondingly, for any $x\in I_n$, we have
\[
\max_{x\in I_n}|h(x)-\tilde{h}_1(x)|\leq \mathcal{O}\Paren{\frac{1}{n\log n}}.
\]
To approximate $h_{na}(x)$,
we approximate $h(x)$ by $\tilde{h}_1(x)$, 
and ${h}(x+(na-1)^{-1})$ by $\tilde{h}_1(x+(na-1)^{-1})$.
The resulting polynomial is
\begin{align*}
\tilde{h}_{na}(x)
&:=-\log\frac{na-1}{na}+(na-1)\Paren{\tilde{h}_1(x+(na-1)^{-1})-\tilde{h}_1(x)}\\
&=-\log\frac{na-1}{c_l a \log n}+(na-1)\Paren{ \sum_{j=0}^{d} b_j \Paren{\frac{n}{c_l\log n}}^{j-1}\Paren{\Paren{x+\frac{1}{na-1}}^j-x^j}}.
\end{align*}
By the above reasoning, the error of $\tilde{h}_{na}$ in approximating $h_{na}$ over $I_n$ satisfies 
\[
\max_{x\in I_n}|h_{na}(x)-\tilde{h}_{na}(x)|\leq \mathcal{O}\Paren{\frac{na}{n\log n}}\leq \mathcal{O}\Paren{\varepsilon}.
\]
Moreover, by Corollary~\ref{Bern_error_h_n}, 
\[
\max_{x\in [0,1/2]}|B_{na}'(h,x)-h_{na}(x)|=\max_{x\in [0,1/2]}|B_{na-1}(h_{na},x)-h_{na}(x)|\leq 1.
\]
The triangle inequality combines the above two inequalities and yields
\[
\max_{x\in I_n}|B_{na}'(h,x)-\tilde{h}_{na}(x)|\leq 1+\mathcal{O}\Paren{\varepsilon}.
\]
Therefore, denoting
\[
\tilde{H}_{na}(x):= \int_0^{x}\tilde{h}_{na}(t) dt,
\]
and noting that $B_{na}(h,0)=0$, we have
\begin{Lemma}\label{Est_Error}
For any $x\in I_n$, 
\[
|B_{na}(h,x)-\tilde{H}_{na}(x)|\leq \int_{0}^{x} |B_{na}'(h,t)-\tilde{h}_{na}(t)| dt\leq x\Paren{1+\mathcal{O}\Paren{\varepsilon}}.
\]
\end{Lemma}

\section{A competitive entropy estimator}\label{Est_const}
In this section, we design an explicit entropy estimator $\hat{H}$ based on $\tilde{H}_{na}$ and the empirical estimator.
Note that $\tilde{H}_{na}(x)$ is a polynomial with zero constant term. For $t\geq 1$, denote
\[
g_t:=\sum_{j=t}^{d} \frac{b_j}{j+1}\Paren{\frac{n}{c_l\log n}}^{j-1} \Paren{\frac{1}{na-1}}^{j-t} \binom{j+1}{j-t+1}.
\]
Setting $b_t'=g_t$ for $t\geq 2$ and $b_1'={g_{1}}-{\log\frac{na-1}{c_l a \log n}}$, we have the following lemma.
\begin{Lemma}\label{approx_coeff}
The function $\tilde{H}_{na}(x)$ can be written as
\[
\tilde{H}_{na}(x) =  \sum_{t=1}^{d} b_t' x^t.
\]
In addition, its coefficients satisfy
\[
|b_t'|\leq \Paren{\frac{n}{c_l\log n}}^{t-1} \mathcal{O}(2^{4d}).
\]
\end{Lemma}
The proof of the above lemma is delayed to the end of this section.

To simplify our analysis and remove the dependency between the counts $N_i$, we use the conventional \emph{Poisson sampling} technique~\cite{mmentro,mmcover}. Specifically, instead of drawing exactly $n$ samples, we make the sample size an independent Poisson random variable $N$ with mean $n$. This does not change the statistical natural of the problem as $N\sim \Poi(n)$ highly concentrates around its mean (see Lemma~\ref{tailprob}). We still define $N_i$ as the counts of symbol $i$ in $X^N$. Due to Poisson sampling, these counts are now independent, and satisfy $N_i\sim\Poi(np_i), \ \forall i\in[k]$.

For each $i\in[k]$, let $N_i^{\underline{t}}:=\prod_{m=0}^{t-1}(N_i-m)$ be the order-$t$ falling factorial of $N_i$. The following identity is well-known:
\[
\EE[N_i^{\underline{t}}]=(np_i)^t,\ \forall t\leq n.
\]
Note that for sufficiently small $c_s$, the degree parameter $d=c_s\log n\leq n, \forall n$.
By the linearity of expectation, the unbiased estimator of $\tilde{H}_{na}(p_i)$ is 
\[
\hat{H}_{na}(N_i):= \sum_{t=1}^{d} b_t' \frac{N_i^{\underline{t}}}{n^t}.
\]
Let $N'$ be an independent Poisson random variable with mean $n$, and $X^{N'}$ be an independent length-${N'}$ sample sequence drawn from $\vec{p}$. Analogously, we denote by $N_i'$ the number of times that symbol $i\in[k]$ appears. Depending on whether $N_i'> \varepsilon^{-1}$ or not, we classify $p_i, i\in[k]$ into two categories: small- and large- probabilities. For small probabilities, we apply a simple variant of $\hat{H}_{na}(N_i)$; for large probabilities, we estimate $h(p_i)$ by essentially the empirical estimator. Specifically, for each $i\in[k]$, we estimate $h(p_i)$ by 
\[
\hat{h}(N_i,N_i'):= \hat{H}_{na}(N_i)\cdot \indic_{N_i\leq c_l\log n}\cdot \indic_{N_i'\leq \varepsilon^{-1}}+h\Paren{\frac{N_i}{n}}\cdot\indic_{N_i'> \varepsilon^{-1}}.
\]
Consequently, we approximate $H(\vec{p})$ by
\[
\hat{H}(X^N,X^{N'}):= \sum_{i\in[k]} \hat{h}(N_i,N_i').
\]
For the simplicity of illustration, we will refer to 
\[
\hat{H}_S(X^N,X^{N'}):=\sum_{i\in[k]} \hat{H}_{na}(N_i)\cdot \indic_{N_i\leq c_l\log n}\cdot \indic_{N_i'\leq \varepsilon^{-1}}
\] 
as the \emph{small-probability estimator}, and 
\[
\hat{H}_L(X^N,X^{N'}):=\sum_{i\in[k]} h\Paren{\frac{N_i}{n}}\cdot\indic_{N_i'> \varepsilon^{-1}}
\] 
as the \emph{large-probability estimator}. Clearly, $\hat{H}$ is the sum of these two estimators.

In the next two sections, we analyze the bias and mean absolute deviation of $\hat{H}$. In Section~\ref{Bias_bound}, we show that for any $\vec{p}$, the absolute bias of $\hat{H}$ satisfies
\[
\Abs{\EE[\hat{H}(X^N,X^{N'})]-H(\vec{p})}\leq \Abs{\text{Bias}(\hat{H}^E, na)}+\Paren{1+\mathcal{O}\Paren{\varepsilon}}\Paren{1\land(\varepsilon^{-1}+1)\frac{S_{\vec{p}}}{n}}.
\]
In Section~\ref{Mean_H}, we show that the mean absolute deviation of $\hat{H}$ satisfies 
\[
\EE\Abs{\hat{H}(X^N,X^{N'})-\EE[\hat{H}(X^N,X^{N'})]}\leq \mathcal{O}\Paren{\frac{1}{n^{1-\Theta(c_s)}}}.
\]
For sufficiently small $c_s$, the triangle inequality combines the above inequalities and yields 
\[
\EE\Abs{\hat{H}(X^N,X^{N'})-H(\vec{p})}\leq \Abs{\text{Bias}(\hat{H}^E, na)}+\Paren{1+c\cdot \varepsilon}\land \Paren{\frac{S_{\vec{p}}}{\varepsilon n}+\frac{1}{n^{0.49}}}.
\]
This basically completes the proof of Theorem~\ref{thm5}. 
\subsection*{Proof of Lemma~\ref{approx_coeff}}
We begin by proving the first claim:
\[
\tilde{H}_{na}(x) =  -\sum_{t=1}^{d} b_t' x^t.
\]
By definition, $\tilde{H}_{na}(x)$ satisfies
\begin{align*}
&\tilde{H}_{na}(x)+\Paren{\log\frac{na-1}{c_l a \log n}}x\\
&=(na-1)\Paren{\sum_{j=1}^{d} \frac{b_j}{j+1} \Paren{\frac{n}{c_l\log n}}^{j-1}\Paren{\Paren{x+\frac{1}{na-1}}^{j+1}-\Paren{\frac{1}{na-1}}^{j+1}-x^{j+1}}}\\
&=\sum_{j=1}^{d} \frac{b_j}{j+1} \Paren{\frac{n}{c_l\log n}}^{j-1}\Paren{\sum_{m=0}^{j-1}\Paren{\frac{1}{na-1}}^{m}x^{j-m}\binom{j+1}{m+1}}\\
&=\sum_{t=1}^{d}x^t \Paren{\sum_{j=t}^{d} \frac{b_j}{j+1}\Paren{\frac{n}{c_l\log n}}^{j-1} \Paren{\frac{1}{na-1}}^{j-t} \binom{j+1}{j-t+1}}.
\end{align*}
The last step follows by reorganizing the indices. 

Next we prove the second claim. Recall that $d = c_s\log n$, thus
\[
\log\frac{na-1}{c_l a \log n}\leq \mathcal{O}(2^{4d}).
\]
Since $b_t'=g_t$ for $t\geq 2$ and $b_1'={g_{1}}-{\log\frac{na-1}{c_l a \log n}}$, it suffices to bound the magnitude of $g_t$:
\begin{align*}
\Abs{g_t}
&\leq \sum_{j=t}^{d} \frac{\Abs{b_j}}{j+1}\Paren{\frac{n}{c_l\log n}}^{j-1} \Paren{\frac{1}{na-1}}^{j-t} \binom{j+1}{j-t+1}\\
&\leq \sum_{j=t}^{d} \Abs{b_j} \Paren{\frac{1}{c_l\log n}}^{j-1} n^{t-1}\binom{j}{t}\\
&\leq \Paren{\frac{n}{c_l\log n}}^{t-1} \sum_{j=t}^{d} \Abs{b_j} \binom{j}{t}\\
&\leq \Paren{\frac{n}{c_l\log n}}^{t-1} \sum_{j=t}^{d} \Abs{b_j} \binom{d}{j-t}\\
&\leq \Paren{\frac{n}{c_l\log n}}^{t-1} \mathcal{O}(2^{4d}).
\end{align*}

\section{Bounding the bias of $\boldsymbol{\hat{H}}$}\label{Bias_bound}
By the triangle inequality, the absolute bias of $\hat{H}$ in estimating $H(\vec{p})$ satisfies
\begin{align*}
\Abs{\sum_{i\in[k]}(\EE[\hat{h}(N_i,N_i')]-h(p_i))}
&\leq\Abs{\sum_{i\in[k]}(B_{na}(h,p_i)-h(p_i))} \\&+\Abs{\sum_{i\in[k]}(\EE[\hat{h}(N_i,N_i')]-B_{na}(h,p_i))}.
\end{align*}

Note that the first term on the right-hand side is the absolute bias of the empirical estimator with sample size $na=\varepsilon n\log n$, i.e.,
\[
\Abs{\text{Bias}(\hat{H}^E, na)}=\Abs{\sum_{i\in[k]}(B_{na}(h,p_i)-h(p_i))}.
\]
Hence, we only need to consider the second term on the right-hand side, which admits
\[
 \Abs{\sum_{i\in[k]}(\EE[\hat{h}(N_i,N_i')]-B_{na}(h,p_i))}\leq \text{Bias}_S+ \text{Bias}_L,
\]
where 
\[
\text{Bias}_S:=\left| \sum_{i\in[k]}\EE\left[\Paren{\hat{H}_{na}(N_i)\cdot \indic_{N_i\leq c_l\log n}-B_{na}(h,p_i)}\cdot \indic_{N_i'\leq \varepsilon^{-1}}\right]\right|
\]
is the absolute bias of the small-probability estimator, and
\[
\text{Bias}_L:=\left|\sum_{i\in[k]}\EE\left[\Paren{h\Paren{\frac{N_i}{n}}-B_{na}(h,p_i)}\cdot\indic_{N_i'> \varepsilon^{-1}}\right]\right|
\]
is the absolute bias of the large-probability estimator.

Assume that $c_l$ is sufficiently large. In Section~\ref{Bias_small}, we bound the small-probability bias by
\[
|\text{Bias}_S|\leq \Paren{1+\mathcal{O}\Paren{\varepsilon}}\Paren{1\land (\varepsilon^{-1}+1)\frac{S_{\vec{p}}}{n}}.
\]
In Section~\ref{Bias_large}, we bound the large-probability bias by
\[
|\text{Bias}_L|\leq 2\Paren{\varepsilon\land \frac{S_{\vec{p}}}{n}}.
\]
\subsection{Bias of the small-probability estimator}\label{Bias_small}
We first consider the quantity $\text{Bias}_S$. By the triangle inequality, 
\begin{align*}
\text{Bias}_S 
&\leq \sum_{i: p_i\not\in I_n}\left|\EE[\hat{H}_{na}(N_i)\cdot \indic_{N_i\leq c_l\log n}]-B_{na}(h,p_i)\right|\cdot \EE[\indic_{N_i'\leq \varepsilon^{-1}}] \\
&+\sum_{i: p_i\in I_n} \left|\EE\left[\hat{H}_{na}(N_i)\right]-B_{na}(h,p_i)\right|\cdot \EE[\indic_{N_i'\leq \varepsilon^{-1}}]\\
&+\sum_{i: p_i\in I_n} \left| \EE\left[\hat{H}_{na}(N_i)\cdot \indic_{N_i> c_l \log n}\right]\cdot \EE\left[\indic_{N_i'\leq \varepsilon^{-1}}\right]\right|.
\end{align*}
Assume $\varepsilon\log n \ge 1$  and consider the first sum on the right-hand side. By the general reasoning in the proof of Lemma~\ref{Bias_extra}, we have the following result:
 \[
 \hat{H}_{na}(N_i)\cdot \indic_{N_i\leq c_l\log n}\leq \mathcal{O}(2^{5d}) \frac{\log^2 n}{n}.
 \]
Further assume that $c_s$ and $c_l$ are sufficiently small and large, respectively. For large enough $n$, the above inequality bounds the first sum by
\[
\sum_{i: p_i\not\in I_n}\left|{\hat{H}_{na}(N_i)\cdot \indic_{N_i\leq c_l\log n}-B_{na}(h,p_i)}\right|\cdot \EE[\indic_{N_i'\leq \varepsilon^{-1}}]
\leq \sum_{i: p_i\not\in I_n}\EE[\indic_{N_i'\leq \varepsilon^{-1}}]
\leq \frac{1}{n^5}\cdot \frac{n}{c_l\log n}\leq \frac{1}{n^4}.
\]
For the second sum on the right-hand side, by Lemma~\ref{Est_Error},
\begin{align*}
\sum_{i: p_i\in I_n} \left|\EE\left[\hat{H}_{na}(N_i)\right]-B_{na}(h,p_i)\right|\cdot \EE[\indic_{N_i'\leq \varepsilon^{-1}}]
& \leq \sum_{i: p_i\in I_n}  \left|\EE\left[\hat{H}_{na}(N_i)\right]-B_{na}(h,p_i)\right|\cdot \EE[\indic_{N_i'\leq \varepsilon^{-1}}]\\
& =\sum_{i: p_i\in I_n}   \left|\tilde{H}_{na}(p_i)-B_{na}(h,p_i)\right|\cdot \EE[\indic_{N_i'\leq \varepsilon^{-1}}]\\
& \leq \sum_{i: p_i\in I_n}  \Paren{1+\mathcal{O}\Paren{\varepsilon}}  p_i\cdot \EE[\indic_{N_i'\leq \varepsilon^{-1}}]\\
& \leq \Paren{1+\mathcal{O}\Paren{\varepsilon}}\Paren{1\land (\varepsilon^{-1}+1)\frac{S_{\vec{p}}}{n}}.
\end{align*}
The following lemma bounds the last sum and completes our argument.
\begin{Lemma}\label{Bias_extra}
For sufficiently large $c_l$,
\[
\sum_{i\in[k]} \left| \EE\left[\hat{H}_{na}(N_i)\cdot \indic_{N_i> c_l\log n}\right]\cdot \EE\left[\indic_{N_i'\leq \varepsilon^{-1}}\right]\right|\leq \frac{1}{n^5}.
\]
\end{Lemma}
\begin{proof}
For simplicity, we assume that $c_l\geq 4$ and $\varepsilon\log n \ge 1$.
By the triangle inequality,
\begin{align*}
&\left| \EE\left[\hat{H}_{na}(N_i)\cdot \indic_{N_i> c_l\log n}\right]\cdot \EE\left[\indic_{N_i'\leq \varepsilon^{-1}}\right]\right|\\
&\leq \sum_{j=1}^{\infty}\left| \EE\left[\hat{H}_{na}(N_i)\cdot \indic_{c_l(j+1)\log n\geq N_i> c_l j\log n}\right]\cdot \EE\left[\indic_{N_i'\leq \varepsilon^{-1}}\right]\right|.
\end{align*}
To bound the last term, we need the following result: for $j\geq 1$,
\[
\left| \EE\left[\indic_{c_l(j+1)\log n\geq N_i> c_l j\log n}\right]\cdot \EE\left[\indic_{N_i'\leq \varepsilon^{-1}}\right]\right|\leq \Paren{1+\varepsilon^{-1}}np_i\cdot e^{-\Theta(c_l j \log n)}.
\]
To prove this inequality, we apply Lemma~\ref{tailprob} and consider two cases:\\
Case 1: If $np_i<({3c_l}/{8})j\log n$, then 
\[
\EE\left[\indic_{c_l (j+1)\log n\geq N_i> c_l j\log n}\right]\leq np_i\cdot e^{-\Theta(c_l j \log n)}.
\]
Case 2: If $np_i\geq({3c_l}/{8})j\log n$, then
\[
\EE\left[\indic_{N_i'\leq \varepsilon^{-1} }\right]\leq np_i \varepsilon^{-1}\cdot e^{-\Theta(c_l j \log n)}.
\]
This essentially completes the proof.
Next we bound $\hat{H}_{na}(N_i)$ for $N_i \in [c_l j\log n, c_l (j+1)\log n]$:
\begin{align*}
|\hat{H}_{na}(N_i)|
&= \left|\Paren{\log\frac{na-1}{c_l a \log n}}\frac{N_i}{n} +\sum_{t=1}^{d} b_t' \frac{N_i^{\underline{t}}}{n^t}\right|\\
&\leq \mathcal{O}(2^{4d}) \sum_{t=1}^{c_s\log n} \Paren{\frac{n}{c_l\log n}}^{t-1} \frac{(c_l (j+1)\log n)^t}{n^t}\\
&\leq \mathcal{O}(2^{5d}) \frac{c_l j \log n}{n}\sum_{t=1}^{c_s\log n} j^{t-1}\\
&\leq \mathcal{O}(2^{5d}) \frac{c_l j \log n}{n}(j^{c_s\log n}+c_s\log n).
\end{align*}
Hence, for sufficiently large $c_l$,
\begin{align*}
&\left| \EE\left[\hat{H}_{na}(N_i)\cdot \indic_{N_i> c_l \log n}\right]\cdot \EE\left[\indic_{N_i'\leq \varepsilon^{-1}}\right]\right|\\
&\leq \sum_{j=1}^{\infty}\left| \EE\left[\hat{H}_{na}(N_i)\cdot \indic_{c_l (j+1)\log n\geq N_i> c_l j\log n}\right]\cdot \EE\left[\indic_{N_i'\leq \varepsilon^{-1}}\right]\right|\\
&\leq   \sum_{j=1}^{\infty}\mathcal{O}(2^{5d}) \cdot c_l j \log n(j^{c_s\log n}+c_s\log n)\cdot \EE\left[\indic_{c_l (j+1)\log n\geq N_i> c_l j\log n}\right]\cdot \EE\left[\indic_{N_i'\leq \varepsilon^{-1}}\right]\\
&\leq \mathcal{O}(2^{5d})  \sum_{j=1}^{\infty}\Paren{1+\varepsilon^{-1}}p_i\cdot e^{-\Theta(c_l j \log n)}\cdot c_l j \log n(j^{c_s\log n}+c_s\log n)\\
&\leq p_i\sum_{j=1}^{\infty}{\frac{1}{2n^{5j}}}\\
&\leq{\frac{p_i}{n^5}}.
\end{align*}
This yields the desired result. \qedhere
\end{proof}

\subsection{Bias of the large-probability estimator}\label{Bias_large}
In this section we prove the bound $|\text{Bias}_L|\leq 2\Paren{\varepsilon\land ({S_{\vec{p}}}/{n}})$. By the triangle inequality, 
\begin{align*}
\text{Bias}_L
&\leq \sum_{i\in[k]}\left|\EE\left[h\Paren{\frac{N_i}{n}}-B_{na}(h,p_i)\right]\right|\cdot \EE\left[\indic_{N_i'> \varepsilon^{-1}}\right]\\
&\leq \sum_{i\in[k]}\left|h(p_i)-B_{na}(h,p_i)\right|\cdot \EE\left[\indic_{N_i'> \varepsilon^{-1}}\right]
+\sum_{i\in[k]}\left|\EE\left[h\Paren{\frac{N_i}{n}}-h(p_i)\right]\right|\cdot \EE\left[\indic_{N_i'> \varepsilon^{-1}}\right].
\end{align*}
We need the following inequality to bound the right-hand side.
\begin{align*}
0
\leq x\log x-(x-1)
\leq (x-1)^2,\ \forall x\in [0,1].
\end{align*}
For simplicity, denote $\hat{p}_i:=N_i/n$. Then,
\begin{align*}
\Abs{\EE\left[h\Paren{\frac{N_i}{n}}-h(p_i)\right]}
&=\Abs{\EE[p_i\log p_i-\hat{p}_i\log \hat{p}_i]}\\
&\leq \Abs{\EE[p_i\log p_i-\hat{p}_i\log p_i]}+\Abs{\EE[\hat{p}_i\log p_i-\hat{p}_i\log \hat{p}_i]}\\
&=p_i\Abs{\EE\left[\frac{\hat{p}_i}{p_i}\log \frac{\hat{p}_i}{p_i} \right]}\\
&\leq p_i\Abs{\EE\left[\Paren{\frac{\hat{p}_i}{p_i}-1}+\Paren{\frac{\hat{p}_i}{p_i}-1}^2 \right]}\\
&= \frac{1}{n}.
\end{align*}
The above derivations also proved that
\[
\left|h(p_i)-B_{na}(h,p_i)\right|\leq \frac{1}{na}.
\]
Consider the first term on the right-hand side. By the above bounds and the Markov's inequality,
\begin{align*}
\sum_{i\in[k]}\left|h(p_i)-B_{na}(h,p_i)\right|\cdot \EE\left[\indic_{N_i'> \varepsilon^{-1}}\right]
& \leq \frac{1}{na}\sum_{i\in[k]}\EE\left[\indic_{N_i'> \varepsilon^{-1}}\right]\\
&\leq \frac{1}{na}\sum_{i\in[k]}\Paren{\indic_{p_i>0}\land \varepsilon np_i}\\
&\leq {\varepsilon\land \frac{S_{\vec{p}}}{n}}.
\end{align*}
For the second term, an analogous argument yields
\[
\sum_{i\in[k]}\left|\EE\left[h\Paren{\frac{N_i}{n}}-h(p_i)\right]\right|\cdot \EE\left[\indic_{N_i'> \varepsilon}\right]\leq {\varepsilon\land \frac{S_{\vec{p}}}{n}}.
\]

\section{Bounding the mean absolute deviation of $\boldsymbol{\hat{H}}$}\label{Mean_H}
By the Jensen's inequality, 
\[
\EE[|\hat{H}(X^N,X^{N'})-\EE[\hat{H}(X^N,X^{N'})]|]\leq \sqrt{\Var{(\hat{H}(X^N,X^{N'}))}}.
\]
Hence, to bound the mean absolute deviation of $\hat{H}$, we only need to bound its variance.
Note that the counts are mutually independent. The inequality $\Var(X+Y)\leq 2(\Var(X)+\Var(Y))$ implies
\[
\Var{(\hat{H}(X^N,X^{N'}))}= \sum_{i\in[k]} \Var(\hat{h}(N_i,N_i'))\leq 2\Var_S+2\Var_L,
\]
where 
\[
\Var_S:= \sum_{i\in[k]} \Var\Paren{\hat{H}_{na}(N_i)\cdot \indic_{N_i\leq c_l \log n}\cdot \indic_{N_i'\leq \varepsilon^{-1}}}
\]
is the variance of the small-probability estimator, and
\[
\Var_L:= \sum_{i\in[k]} \Var\Paren{h\Paren{\frac{N_i}{n}}\cdot\indic_{N_i'> \varepsilon^{-1}}}
\]
is the variance of the large-probability estimator. Assume that $c_l$ and $c_s$ are sufficiently large and small constants, respectively.
In Section~\ref{Var_small}, we prove
\[
\Var_S\leq \mathcal{O}\Paren{\frac{1}{n^{1-\Theta(c_s)}}},
\]
and in Section~\ref{Var_large}, we show
\[
\Var_L\leq \mathcal{O}\Paren{\frac{(\log n)^3}{n}}.
\]
\subsection{Variance of the small-probability estimator}\label{Var_small}
First we bound the quantity $\Var_S$. Our objective is to prove $\Var_S\leq \mathcal{O}\Paren{{1}/{n^{1-\Theta(c_s)}}}$. According to the previous derivations,
\begin{align*}
\Var_S
&\leq 2\sum_{i\in[k]} \Var\Paren{\hat{H}_{na}(N_i)\cdot \indic_{N_i> c_l\log n}\cdot \indic_{N_i'\leq \varepsilon^{-1}}}\\
&+2\sum_{i\in[k]} \Var\Paren{\hat{H}_{na}(N_i)\cdot \indic_{N_i'\leq \varepsilon^{-1}}}\\
&\leq 2\sum_{i\in[k]} \EE[(\hat{H}_{na}(N_i))^2\cdot \indic_{N_i>c_l\log n}]\cdot \EE[\indic_{N_i'\leq \varepsilon^{-1}}]\\
&+2\sum_{i\in[k]} \Var\Paren{\hat{H}_{na}(N_i)}\cdot \EE[\indic_{N_i'\leq \varepsilon^{-1}}]+2\sum_{i\in[k]} (\EE[\hat{H}_{na}(N_i)])^2\cdot \Var(\indic_{N_i'\leq \varepsilon^{-1}})\\
&\leq 2\sum_{i\in[k]} \EE[(\hat{H}_{na}(N_i))^2\cdot \indic_{N_i>c_l\log n}]\cdot \EE[\indic_{N_i'\leq \varepsilon^{-1}}]\\
&+2\sum_{i\in[k]} \Var\Paren{\hat{H}_{na}(N_i)}\cdot \EE[\indic_{N_i'\leq \varepsilon^{-1}}]+2\sum_{i\in[k]} (\tilde{H}_{na}(p_i))^2\cdot \Var(\indic_{N_i'\leq \varepsilon^{-1}}),
\end{align*}
where the first step follows from the inequality $\Var(X-Y)\leq 2(\Var(X)+\Var(Y))$, the second step follows from $\Var(A\cdot B)=\EE[A^2]\Var(B)+\Var(A)(\EE[B])^2$ for $A\perp B$, and the last step follows from $\EE[\hat{H}_{na}(N_i)]=\tilde{H}_{na}(p_i)$.

For the first term on the right-hand side, similar to the proof of Lemma~\ref{Bias_extra}, we have
\[
\sum_{i\in[k]} \left| \EE\left[(\hat{H}_{na}(N_i))^2\cdot \indic_{N_i> c_l\log n}\right]\cdot \EE\left[\indic_{N_i'\leq \varepsilon^{-1}}\right]\right|\leq  \sum_{i\in[k]} {\frac{p_i}{n^3}}p_i = \frac{1}{n^3},
\]
for sufficiently large $c_l$.

For the second term on the right-hand side,
\begin{align*}
&\sum_{i\in[k]} \Var\Paren{\hat{H}_{na}(N_i)}\cdot \EE[\indic_{N_i'\leq \varepsilon^{-1}}]\\
&\leq \mathcal{O}(2^{8d}) \sum_{i\in[k]} d\sum_{t=1}^{d} \Paren{\frac{n}{c_l\log n}}^{2(t-1)} \frac{\Var(N_i^{\underline{t}})}{n^{2t}}\cdot \EE[\indic_{N_i'\leq \varepsilon^{-1}}]\\
&=\mathcal{O}(2^{8d}) \frac{d}{n^2}\sum_{i\in[k]} \sum_{t=1}^{d} \Paren{\frac{1}{c_l\log n}}^{2(t-1)} \Var(N_i^{\underline{t}})\cdot \EE[\indic_{N_i'\leq \varepsilon^{-1}}]\\
&=\mathcal{O}(2^{8d}) \frac{d}{n^2}\sum_{i\in[k]} \sum_{t=1}^{d} \Paren{\frac{1}{c_l\log n}}^{2(t-1)} (np_i)^t \sum_{j=0}^{t-1} \binom{t}{j}(np_i)^j\frac{t!}{j!}\cdot \EE[\indic_{N_i'\leq \varepsilon^{-1}}]\\
&\leq \mathcal{O}(2^{8d}) \frac{d}{n^2}\sum_{i\in[k]} \sum_{t=1}^{d} \Paren{\frac{1}{c_l\log n}}^{2(t-1)} (np_i)^t (t+np_i)^t \cdot \EE[\indic_{N_i'\leq \varepsilon^{-1}}]\\
&\leq \mathcal{O}(2^{8d}) \frac{d}{n^2}\sum_{i\in[k]} \sum_{t=1}^{d} \Paren{\frac{1}{c_l\log n}}^{2(t-1)} 2^t((np_i)^{2t}+(np_i)^t t^t) \cdot \Pr(N_i'\leq \varepsilon^{-1})\\
&\leq \mathcal{O}(2^{8d}) \frac{d}{n}\sum_{i\in[k]} p_i\sum_{t=1}^{d} \Paren{\frac{1}{c_l\log n}}^{2(t-1)} 2^t\left((\varepsilon^{-1} + 2t)^{2t-1}\cdot \Pr(N_i'\leq \varepsilon^{-1}+2t)\right.\\
&+\left.(\varepsilon^{-1}+t)^{t-1} t^t \cdot \Pr(N_i'\leq \varepsilon^{-1}+t)\right)\\
&\leq \mathcal{O}(2^{9d}) \frac{d}{n}.
\end{align*}

It remains to bound the third term. Noting that $|\tilde{H}_{na}(p_i)|\leq \mathcal{O}(2^{5d})p_i$, we have
\begin{align*}
&\sum_{i\in[k]} (\tilde{H}_{na}(p_i))^2\cdot \Var(\indic_{N_i'\leq \varepsilon^{-1}})\\
&\leq \mathcal{O}(2^{8d}) \sum_{i\in[k]} \sum_{t=1}^{d} \Paren{\frac{n}{c_l\log n}}^{2(t-1)} p_i^{2t} \cdot \Var(\indic_{N_i'\leq \varepsilon^{-1}})\\
&\leq \mathcal{O}(2^{8d}) \sum_{i\in[k]} \sum_{t=1}^{d} \Paren{\frac{n}{c_l\log n}}^{2(t-1)} p_i^{2t} \cdot \Pr(N_i'\leq \varepsilon^{-1})\\
&= \mathcal{O}(2^{8d}) \sum_{i\in[k]} p_i \sum_{t=1}^{d} \Paren{\frac{n}{c_l\log n}}^{2(t-1)} p_i^{2t-1} \cdot  \sum_{m = 0}^{\varepsilon^{-1}} e^{-np_i} \frac{(np_i)^m}{m!}\\
&\leq  \mathcal{O}(2^{8d}) \sum_{i\in[k]} p_i \sum_{t=1}^{d} \Paren{\frac{n}{c_l\log n}}^{2(t-1)} \Paren{\frac{2t-1+\varepsilon^{-1}}{n}}^{2t-1}\cdot \Pr(N_i\leq 2t-1+\varepsilon^{-1})\\
&\leq  \mathcal{O}(2^{8d}) \sum_{i\in[k]} p_i \cdot c_s\log n \frac{c_l\log n}{n}\\
&\leq \mathcal{O}\Paren{\frac{2^{9d}}{n}}.
\end{align*}
Consolidating all the three bounds above yields
\[
\Var_S\leq \frac{2}{n^3}+\mathcal{O}(2^{9d}) \frac{d}{n}+\mathcal{O}\Paren{\frac{2^{9d}}{n}}\leq \frac{1}{n^{1-\Theta(c_s)}},
\]
where the last step follows by $d=c_s\log n$.
\subsection{Variance of the large-probability estimator}\label{Var_large}
In this section we bound the quantity $\Var_L$. Our objective is to prove $\Var_L\leq \mathcal{O}((\log n)^3/n)$. Due to independence,
\begin{align*}
\Var_L=&\sum_{i\in[k]} \Var\Paren{h\Paren{\frac{N_i}{n}}\cdot\indic_{N_i'> \varepsilon^{-1}}}.
\end{align*}
The following lemma bounds the last sum.
\begin{Lemma}
For $s\geq{1}$, 
\[
\sum_{i\in[k]} \Var\Paren{h\Paren{\frac{N_i}{n}}\cdot\indic_{N_i'> s}} \leq  (\log n)^2\frac{4s}{n}.
\]
\end{Lemma}

\begin{proof}
Decompose the variances,
\begin{align*}
\sum_{i\in[k]} \Var{\Paren{h \left(\frac{N_i}{n} \right) \indic_{N_i'> s}}}
& =\Var(\indic_{N_i'> s})\EE\left[h^2 \left(\frac{N_i}{n} \right)\right]+ \sum_{i\in[k]} \Paren{\EE[\indic_{N_i'> s}]}^2\Var{\Paren{h \left(\frac{N_i}{n} \right)}}\\
&\leq\Var(\indic_{N_i'> s})\EE\left[h^2 \left(\frac{N_i}{n} \right)\right]+ \sum_{i\in[k]} \Var{\Paren{h \left(\frac{N_i}{n} \right)}}.
\end{align*}
To bound the first term on the right-hand side,
\begin{align*}
\Var(\indic_{N_i'> s})\EE\left[h^2 \left(\frac{N_i}{n} \right)\right]
&\leq \Var(\indic_{N_i'> s})\EE\left[(\log n)^2\Paren{\frac{N_i}{n}}^2\right]\\
&\leq (\log n)^2\frac{p_i}{n}\Paren{1+np_i\Var(\indic_{N_i'> s})},
\end{align*}
where we can further bound 
\begin{align*}
p_i \Var(\indic_{N_i'> s})
&\leq p_i \PP[N_i'\leq s]\\
&= e^{-np_i}\sum_{j=0}^{s}\frac{(np_i)^{j+1}}{(j+1)!}\frac{j+1}{n}\\
&\leq \frac{s+1}{n} e^{-np_i}\sum_{j=0}^{s}\frac{(np_i)^{j+1}}{(j+1)!}\\
&= \frac{s+1}{n} \PP(1\leq N'_x\leq{s+1})\\
&\leq \frac{s+1}{n}.
\end{align*}
To bound the second term, let $\hat{N}_i$
be an i.i.d. copy of $N_i$ for each $i$, 
\begin{align*}
2\Var{\Paren{h \left(\frac{N_i}{n} \right)}}
&= \Var{\Paren{h \left(\frac{N_i}{n} \right) -h\left(\frac{\hat{N}_i}{n} \right)}}\\
&= \EE\left[\Paren{h \left(\frac{N_i}{n} \right) -h\left(\frac{\hat{N}_i}{n} \right)}^2\right]\\
&\leq \EE\left[(\log n)^2\Paren{\frac{N_i}{n} - \frac{\hat{N}_i}{n} }^2\right]\\
&=2(\log n)^2\frac{p_i}{n}.
\end{align*}
A simple combination of these bounds yields the lemma. \qedhere
\end{proof}
Setting $s=\varepsilon^{-1}$ in the above lemma and assuming $\varepsilon\log n\geq 1$, we get
\[
\Var_L= \sum_{i\in[k]} \Var\Paren{h\Paren{\frac{N_i}{n}}\cdot\indic_{N_i'> \varepsilon^{-1}}} \leq  \frac{4(\log n)^3}{n}.
\]

\pagebreak
\section{Experiments}\label{experiments}

We demonstrate the efficacy of the proposed estimators by comparing their performance to
several state-of-the-art estimators~\cite{mmcover, mmentro,mmsize}, and empirical 
estimators with larger sample sizes. Due to similarity of the methods,
we present only the results for
Shannon entropy and support size. For each property, we considered nine natural synthetic
distributions: uniform, two-steps-,  Zipf(1/2), Zipf(1), binomial, geometric, Poisson, Dirichlet(1)-drawn-, and Dirichlet(1/2)-drawn-. 
The plots are shown in Figures~\ref{fig: entropy} and~\ref{fig: support}.

As Theorem~\ref{thm1} and~\ref{thm3} would imply and the experiments confirmed, for both properties, the
proposed estimators with $n$ samples achieved the same accuracy as the
empirical estimators with $n\log n$ samples for Shannon entropy 
and $n\log S_{\vec{p}}$ samples for support size. In particular, for
Shannon entropy, the proposed estimator with $n$ samples performed
significantly better than the $n\log n$-sample empirical estimator,
for all tested distributions and all values of $n$. For both
properties, the proposed estimators are essentially the best among all
state-of-the-art estimators in terms of accuracy and stability.

Next, we describe the experimental settings. 

\subsection*{Experimental settings}
We experimented with nine distributions:
uniform;
a two-steps distribution with probability values $0.5k^{-1}$ and $1.5k^{-1}$; 
Zipf distribution with power $1/2$; 
Zipf distribution with power $1$;
binomial distribution with success probability $0.3$;
geometric distribution with success probability $0.9$;
Poisson distribution with mean $0.3 k$;
a distribution randomly generated from Dirichlet prior with parameter 1;
and a distribution randomly generated from Dirichlet prior with parameter 1/2.

All distributions have support size $k=1000$.
The geometric, Poisson, and Zipf distributions were truncated at $k$
and re-normalized. 
The horizontal axis shows the number of samples, $n$, ranging from $5$ to $640$. 
Each experiment was repeated 100 times and the reported results,
shown on the vertical axis, reflect their mean values and standard deviations. Specifically, the true property value is drawn as a dashed black line, and the other estimators are color coded, with the solid line displaying their mean estimate, and the shaded area corresponding to one standard deviation.

We compared the estimators' performance with $n$ samples to that
of four other recent estimators as well as 
the empirical estimator with $n$, $n\sqrt{\log A}$, and $n\log A$
samples,
where for Shannon entropy, $A=n$ and for support size, $A=S_{\vec p}$.
We chose the parameter $\varepsilon=1$. 
The graphs denote
our proposed estimator by Proposed,
$\hat{F}^E$ with $n$ samples by Empirical,
$\hat{F}^E$ with $n\sqrt{\log{A}}$ samples by Empirical+,
$\hat{F}^E$ with $n\log{A}$ samples by Empirical++,
the profile maximum likelihood estimator (for entropy and support size) in~\cite{mmcover} by PML,
the support-size estimator in~\cite{mmsize} and 
the entropy estimator in~\cite{mmentro} by WY.  
Additional estimators for both properties were compared
in~\cite{mmentro,mmsize,pnas,jvhw} and found to perform
similarly to or worse than the estimators we tested, hence we exclude
them here.

\begin{figure}[p]
\hspace*{-1.8cm}  
\includegraphics[scale=0.41]{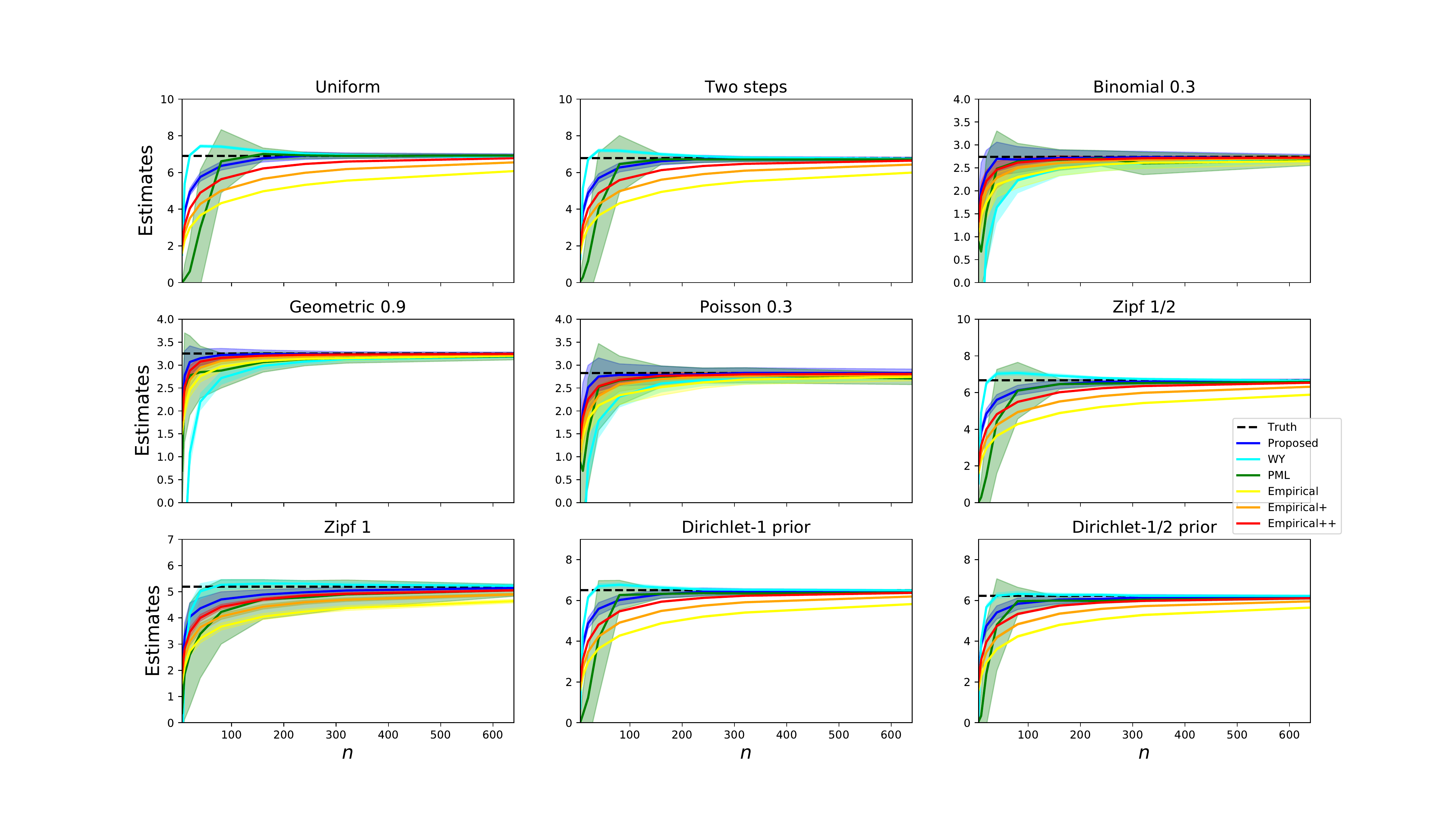}
\vspace*{-1cm}  
\caption{Shannon entropy}
\label{fig: entropy}
\end{figure}
\begin{figure}[p]
\hspace*{-1.8cm}  
\includegraphics[scale=0.41]{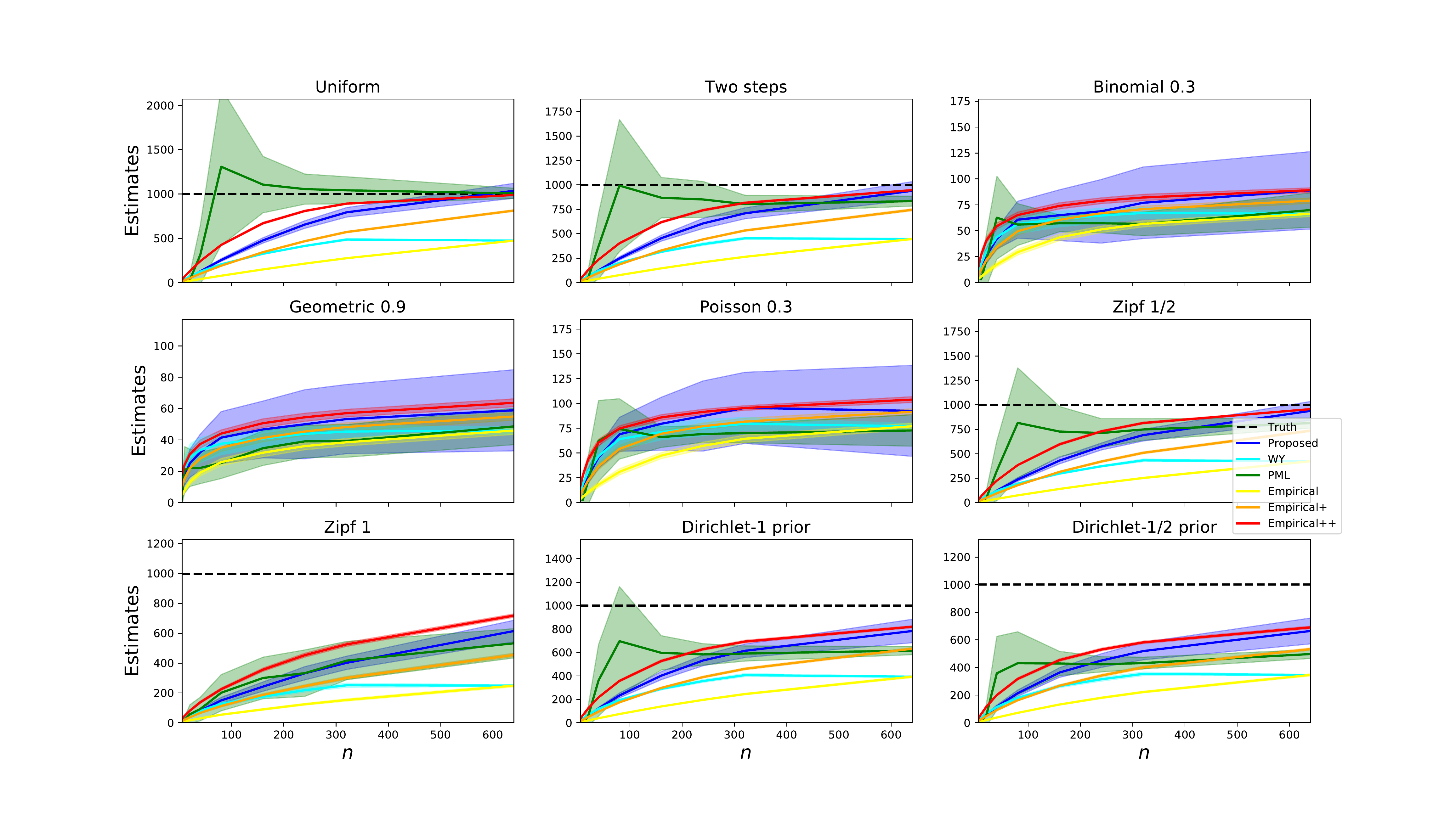}
\vspace*{-1cm}  
\caption{Support size}
\label{fig: support}
\end{figure}

\vfill
\pagebreak
\appendix
\section{A refined estimator for Shannon entropy}\label{refined_est}
For $z\in [0,\infty]$, we define the following two $f$-functions:
\[
f_{1}(z):=-e^{-z}\sum_{j=1}^{\infty}\frac{z^j}{j!}j\log j
\]
and 
\[
f_{2}(z):=-e^{-z}\sum_{j=1}^{\infty}\frac{z^j}{j!}(j+1)\log(j+1).
\]

\subsection{Relating the $\boldsymbol{f}$-functions to Bernstein approximation errors}
For $x\in[0,1]$, set $z=nx$. 
The following lemma relates ${f_{1}}(z)$ and ${f_{2}}(z)$ to $h_{n+1}(x)-B_n(h_{n+1}, x)$.
\begin{Lemma}\label{app_gz}
For $x\in[0,{\log^4 n}/{n}]$, 
\[
h_{n+1}(x)-B_n(h_{n+1}, x)=(h(z+1)-{f_{2}}(z))-(h(z)-{f_{1}}(z))+\tilde{\mathcal{O}}\Paren{\frac{1}{n}}.
\]
\end{Lemma}
\begin{proof}
Let $h_{-1}(x):=h(x+n^{-1})$.
By the linearity of expectation, 
\begin{align*}
h_{n+1}(x)-B_n(h_{n+1}, x)
&= n\Paren{h_{-1}(x)-h(x)-B_n(h_{-1},x)+B_n(h,x)}\\
&= n\Paren{h_{-1}(x)-B_n(h_{-1},x)}-n\Paren{h(x)-B_n(h,x)}.
\end{align*}
Recall that $z=nx$, which implies $z\in[0, \log^4 n]$. We have
\begin{align*}
n\Paren{h_{-1}(x)-B_n(h_1,x)}
&=-(nx+1)\log\Paren{\frac{nx+1}{n}}+\sum_{j=0}^{n} (j+1)\log\Paren{\frac{j+1}{n}} \binom{n}{j} x^j (1-x)^{n-j}\\
&=-(z+1)\log\Paren{\frac{z+1}{n}}+\sum_{j=0}^{n} (j+1)\log\Paren{\frac{j+1}{n}} \binom{n}{j} z^j  \frac{(n-z)^{n-j}}{n^n}\\
&=-(z+1)\log\Paren{{z+1}}+\Paren{1-\frac{z}{n}}^{n}\sum_{j=0}^{n} (j+1)\log\Paren{{j+1}} \binom{n}{j} z^j {(n-z)^{-j}}\\
&=-(z+1)\log\Paren{{z+1}}+\Paren{1-\frac{z}{n}}^{n}\sum_{j=0}^{n} (j+1)\log\Paren{{j+1}} \frac{n^{\underline{j}}}{n^j} \frac{z^j}{j!} {\Paren{1-\frac{z}{n}}^{-j}}\\
&=-(z+1)\log\Paren{{z+1}}+e^{-z}\sum_{j=0}^{\infty}\frac{z^j}{j!}(j+1)\log(j+1)+\tilde{\mathcal{O}}\Paren{\frac{1}{n}}\\
&=h(z+1)-{f_{2}}(z)+\tilde{\mathcal{O}}\Paren{\frac{1}{n}}.
\end{align*}
The second last equality is the most non-trivial step. To establish this equality, we need the following the three inequalities. 
\paragraph{Inequality 1:}
\begin{align*}
0\leq &\Paren{1-\frac{z}{n}}^{n}\sum_{j=\log^5 n+1}^{n} (j+1)\log\Paren{{j+1}} \frac{n^{\underline{j}}}{n^j} \frac{z^j}{j!} {\Paren{1-\frac{z}{n}}^{-j}}\\
&=\Paren{1-\frac{z}{n}}^{n}\sum_{j=\log^5 n+1}^{n} (j+1)\log\Paren{{j+1}} \frac{n^{\underline{j}}}{2^j(n-z)^j} \frac{(2z)^j}{j!}\\
&\leq e^{-z}\sum_{j=\log^5 n+1}^{n} (j+1)\log\Paren{{j+1}} \frac{(2z)^j}{j!}\\
&\leq e^{-z}\sum_{j=\log^5 n+1}^{n}  2 j (j-1)\frac{(2z)^j}{j!}\\
&\leq 8 z^2 e^{-z}\sum_{j=\log^5 n-1}^{n}  \frac{(2z)^j}{j!}\\
&\leq 8 (\log^8 n) \Pr(\Poi(2z)\geq \log^5 n-1)\\
&\leq \frac{1}{n}.
\end{align*}

\paragraph{Inequality 2:}
\begin{align*}
0\leq &e^{-z}\sum_{j=\log^5 n+1}^{\infty}\frac{z^j}{j!}(j+1)\log(j+1)= 2(\log^8 n)\Pr(\Poi(2z)\geq \log^5 n-1) \leq \frac{1}{n}. 
\end{align*}

\paragraph{Inequality 3:} For $j\leq \log^5 n$,
\begin{align*}
\Abs{ e^{-z} - \Paren{1-\frac{z}{n}}^{n} \frac{n^{\underline{j}}}{n^j}  {\Paren{1-\frac{z}{n}}^{-j}}}
&= \Abs{ e^{-z} - \Paren{1-\frac{z}{n}}^{n} \frac{n^{\underline{j}}}{(n-z)^j}}\\
&\leq  \Abs{ e^{-z} -\Paren{1-\frac{z}{n}}^{n}}+\Paren{1-\frac{z}{n}}^{n}\Abs{1- \frac{n^{\underline{j}}}{(n-z)^j}}\\
&\leq  e^{-z}\frac{z^2}{n}+e^{-z}\Abs{1- \frac{n^{\underline{j}}}{(n-z)^j}}\\
&\leq  e^{-z}\frac{z^2}{n}+e^{-z}\Paren{\Abs{1- \frac{n^{j}}{(n-z)^j}}\lor \Abs{1- \frac{(n- \log^5 n)^{j}}{(n-z)^j}}}\\
&\leq  e^{-z}\frac{z^2}{n}+e^{-z}\Paren{\Abs {\exp\Paren{\frac{zj}{n-z}}-1}\lor \Abs{\frac{(\log^5 n-z) {j}}{n-z}}}\\
&\leq  e^{-z}\frac{z^2}{n}+e^{-z}\Paren{\Abs {\frac{zj}{n-z(j+1)}}\lor \Abs{\frac{(\log^5 n) {j}}{n-z}}}\\
&\leq e^{-z}\frac{2\log^{10} n}{n}.
\end{align*}
Note that the last inequality implies 
\begin{align*}
&\Abs{e^{-z}\sum_{j=0}^{\log^5 n}\frac{z^j}{j!}(j+1)\log(j+1)-\Paren{1-\frac{z}{n}}^{n}\sum_{j=0}^{\log^5 n} (j+1)\log\Paren{{j+1}} \frac{n^{\underline{j}}}{n^j} \frac{z^j}{j!} {\Paren{1-\frac{z}{n}}^{-j}}}\\
&\leq \frac{2\log^{10} n}{n}\cdot e^{-z}\sum_{j=0}^{\log^5 n}\frac{z^j}{j!} (2j(j-1))\\
&\leq  \frac{2\log^{10} n}{n}\cdot 2z^2\\
&\leq \frac{4\log^{18} n}{n}.
\end{align*}
This together with Inequality 1 and 2 proves the desired equality. 
Similarly, we have
\[
n\Paren{h(x)-B_n(h,x)}=-z\log z+e^{-z}\sum_{j=1}^{\infty}\frac{z^j}{j!}j\log j+\tilde{\mathcal{O}}\Paren{\frac{1}{n}},
\]
which completes the proof.
\end{proof}

For $x\in I_n$, let $z_1=(na-1)x$, then $z_1\in I_n':=[0, ac_l \log n]$. Hence by the above lemma,
\[
h_{na}(x)-B_{na-1}(h_{na}, x)=(h(z_1+1)-{f_{2}}(z_1))-(h(z_1)-{f_{1}}(z_1))+\tilde{\mathcal{O}}\Paren{\frac{1}{n}}.
\]

In the next section, we approximate the function ${f_{1}}(z)$ with a degree-$d$ polynomial over $I_n'$. 

\subsection{\fontsize{11pt}{11pt}\selectfont Approximating $\boldsymbol{{f_{1}}(z)}$}
First consider the function
\[
{f_{1}}(z)=-e^{-z}\sum_{j=1}^{\infty}\frac{z^j}{j!}j\log j.
\]
our objective is to approximate $f_1$ with a low-degree polynomial and bound the corresponding error. 
To do this, we first establish some basic properties of ${f_{1}}(z)$ in the next section.
\subsubsection{Properties of $\boldsymbol{{f_{1}}(z)}$}\label{prop_g}
\paragraph{Property 1:}
The function ${f_{1}}(z)$ is a continuous function over $[0,\infty)$, and ${f_{1}}(0)=0$.

\paragraph{Property 2:}
For all $z\geq 0$, the value of $f_1(z)$ is non-negative.

\paragraph{Property 3:}
Denote
\[
u(y):= (y+2)\log(y+2)+y\log y-2(y+1)\log(y+1).
\]
Then, for $z\geq0$,
\[
{f_{1}}''(z)
=-e^{-z} \sum _{t=0}^{\infty } \frac{ z^{t} }{t!} \cdot u(t).
\]
Furthermore, we have
\[
-\log 4\leq {f_{1}}''(z)<0. 
\]
\begin{proof}
We prove the equality first.
\begin{align*}
-{f_{1}}''(z)
&=e^{-z} \sum _{t=1}^{\infty } \frac{(t-1) t^2 z^{t-2} \log (t)}{t!}-2 e^{-z} \sum _{t=1}^{\infty } \frac{t^2 z^{t-1} \log (t)}{t!}+e^{-z} \sum _{t=1}^{\infty } \frac{t z^t \log (t)}{t!}\\
&=e^{-z} \sum _{t=0}^{\infty } \frac{ z^{t} (t+2) \log (t+2)}{t!}-2 e^{-z} \sum _{t=0}^{\infty } \frac{z^{t} (t+1) \log (t+1)}{t!}+e^{-z} \sum _{t=0}^{\infty } \frac{t z^t \log (t)}{t!}\\
&=e^{-z} \sum _{t=0}^{\infty } \frac{ z^{t} }{t!} \cdot u(t).
\end{align*}

To prove the inequality, we need the following lemma.
\begin{Lemma}\label{u_ineq}
For $t\geq 0$, 
\[
\frac{\log 4}{t+1}\geq u(t)\geq \frac{1}{t+1}.
\]
\end{Lemma}
By Lemma~\ref{u_ineq}, we have
\begin{align*}
0
&<e^{-z} \sum _{t=0}^{\infty } \frac{ z^{t} }{t!} \cdot \Paren{\frac{1}{t+1}}\\
&\leq e^{-z} \sum _{t=0}^{\infty } \frac{ z^{t} }{t!} \cdot u(t)\\
&=-{f_{1}}''(z)\\
&\leq e^{-z} \sum _{t=0}^{\infty } \frac{ z^{t} }{t!} \cdot \frac{\log 4}{t+1}\\
&=(\log 4) \frac{(1-e^{-z})}{z}\\
&\leq \log 4.
\end{align*}
\end{proof}
\paragraph{Property 4:}
For $z>0$,
\[
0\leq \frac{{f_{1}}''(z)}{h''(z)}\leq \log 4.
\]
\begin{proof}
Recall that $h(z)=-z\log z$,
\[
h''(z)=-\frac{1}{z}
\]
and thus
\begin{align*}
0
&\leq \frac{{f_{1}}''(z)}{h''(z)}\\
&=e^{-z} \sum _{t=0}^{\infty } \frac{ z^{t+1} }{t!} \cdot u(t)\\
&\leq e^{-z} \sum _{t=0}^{\infty } \frac{ z^{t+1} }{t!} \cdot \frac{\log 4}{t+1}\\
&\leq (\log 4) (1-e^{-z})\\
&\leq \log 4,
\end{align*}
where we have used Lemma~\ref{u_ineq} in the third step.
\end{proof}
\subsubsection{Moduli of smoothness}
In this section, we introduce some basic results in approximation theory~\cite{moduli}.
For any function $f$ over $[0,1]$, let $\varphi(x)=\sqrt{x(1-x)}$, the first- and second- order Ditzian-Totik moduli of smoothness quantities of $f$ are
\[
w^1_{\varphi}(f,t):=\sup\Brace{|f(u)-f(v)|: 0\leq u,v\leq 1, |u-v|\leq t\cdot\varphi\Paren{\frac{u+v}{2}}},
\]
and
\[
w^2_{\varphi}(f,t):=\sup\Brace{\Abs{f(u)+f(v)-2f\Paren{\frac{u+v}{2}}}: 0\leq u,v\leq 1, |u-v|\leq 2t\cdot\varphi\Paren{\frac{u+v}{2}}},
\]
respectively. For any integer $m\geq 1$ and any function $f$ over $[0,1]$, let $P_m$ be the collection of degree-$m$ polynomials, and 
\[
E_m[f]:=\min_{g\in P_m} \max_{x\in[0,1]}|f(x)-g(x)|
\]
be the maximum approximation error of the degree-$m$ min-max polynomial of $f$.
The relation between the best polynomial-approximation error $E_m[f]$ of a continuous function $f$ and the smoothness quantity $w^2_{\varphi}(f,t)$ is established in the following lemma~\cite{moduli}.
\begin{Lemma} \label{bound_moduli}
There are absolute constants $C_1$ and $C_2$ such that for any continuous function $f$ over $[0,1]$ and any $m>2$,
\[
E_m[f]\leq C_1 w^2_\varphi(f, m^{-1}),
\]
and 
\[
\frac{1}{m^2}\sum_{t=0}^{m}(t+1) E_t[f]\geq C_2 w^2_\varphi(f, m^{-1}).
\]
\end{Lemma}
The above lemma shows that $w^2_{\varphi}(f,\boldsymbol{\cdot})$ essentially characterizes $E_{\boldsymbol{\cdot}}[f]$.

\subsubsection{Bounding the error in approximating $\boldsymbol{{f_{1}}(x)}$}\label{app_f1}
For simplicity, we define $x':=(ac_l\log n)\cdot x$ and consider the following function.
\[
{f_{1'}}(x):={f_{1}}((ac_l\log n)\cdot x).
\]
Approximating ${f_{1}}(x')$ over $I_n'=[0,ac_l\log n]$ is equivalent to approximate ${f_{1'}}(x)$ over the unit interval $[0,1]$.
According to Lemma~\ref{bound_moduli}, to bound $E_d[{f_{1'}}]$, it suffices to bound $w^2_{\varphi}({f_{1'}},\boldsymbol{\cdot})$. Specifically, we know that
\[
\min_{g\in P_d} \max_{x\in I_n'}|{f_{1}}(x)-g(x)|=E_d[{f_{1'}}]\leq C_1 w^2_\varphi({f_{1'}}, d^{-1}).
\]
Note that by definition, $w^2_\varphi({f_{1'}}, d^{-1})$ is the solution to the following optimization problem.
\[
\sup_{u,v} \Abs{{f_{1'}}(u)+{f_{1'}}(v)-2{f_{1'}}\Paren{\frac{u+v}{2}}}
\]
subject to 
\[
0\leq u,v\leq 1, |u-v|\leq \frac{2}{d}\cdot\varphi\Paren{\frac{u+v}{2}}.
\]
Consider the optimization constraints first. Following~\cite{jvhw}, we denote $M:=(u+v)/2$ and $\delta:=d^{-1}\sqrt{1/M-1}$. 
The feasible region can be written as
\[
[M-d^{-1}\sqrt{M(1-M)},M+d^{-1}\sqrt{M(1-M)}]\cap[0,1]=[M-\delta M, M+\delta M]\cap [0,1].
\]
By Property 3 in Section~\ref{prop_g}, ${f_{1}}(x')$, or equivalently ${f_{1'}}(x)$, is a strictly concave function.
Therefore, the maximum of $\Abs{f(u)+f(v)-2f({u+v}/{2})}$ is attained at the boundary of the feasible region. 
Noting that 
\[
M-d^{-1}\sqrt{M(1-M)}\geq 0 \iff M\geq \frac{1}{d^2+1}
\]
and
\[
M+d^{-1}\sqrt{M(1-M)}\leq 1 \iff M\leq \frac{d^2}{d^2+1},
\]
we only need to consider the following three cases:
\paragraph{Case 1:} 
\[
u=0, v= 2M, M\in[0, {1}/{(d^2+1)}].
\]
\paragraph{Case 2:}
\[
u=2M-1, v= 1, M\in[{d^2}/{(d^2+1)}, 1].
\]
\paragraph{Case 3:}
\[
u=M-\delta M, v= M+\delta M, M\in[{1}/{(d^2+1)},{d^2}/{(d^2+1)}].
\]
To facilitate our derivations, we need the following lemma.
\begin{Lemma}\label{second_mean}
Let $f\in C^1([a,b])$ have second order derivative in $(a,b)$. There exists $c\in(a,b)$ such that
\[
f(a)+f(b)-2f\Paren{\frac{a+b}{2}}=\frac{1}{4}(b-a)^2\cdot f''(c).
\]
\end{Lemma}
First consider Case 1. By the above lemma, there exists $c\in(0,{2}/{(d^2+1)})$ such that
\[
\Abs{{f_{1'}}(0)+{f_{1'}}\Paren{\frac{2}{d^2+1}}-2{f_{1'}}\Paren{\frac{1}{d^2+1}}}
\leq \frac{1}{4}\cdot \Paren{\frac{2}{d^2+1}}^2 \Abs{{f_{1'}}''(c)}
=\Paren{\frac{1}{d^2+1}}^2 \Abs{{f_{1'}}''(c)}.
\]
By definition,
\[
|{f_{1'}}''(x)|=|(ac_l\log n)^2 g_{1}''((ac_l\log n)\cdot x)|\leq (\log 4)(ac_l\log n)^2.
\]
Hence, 
\[
\Paren{\frac{1}{d^2+1}}^2 \Abs{{f_{1'}}''(c)}\leq \mathcal{O}\Paren{\varepsilon^{2}}.
\]
This, together with an analogous argument for Case 2, implies that the objective value is bounded by $\mathcal{O}\Paren{\varepsilon^{2}}$ in both cases. It remains to analyze Case 3. We consider two regimes:
\paragraph{Regime 1:} If $M\leq 4/(d^2+1)$, then $|u-v|=2d^{-1}\sqrt{M(1-M)}\leq 4/d^2$. The above derivations again give us 
\[
 \Abs{{f_{1'}}(u)+{f_{1'}}(v)-2{f_{1'}}\Paren{\frac{u+v}{2}}}\leq \mathcal{O}\Paren{\varepsilon^{2}}.
\]
\paragraph{Regime 2:} If $4/(d^2+1)\leq M\leq {d^2}/{(d^2+1)}$, then 
\[
M-\delta M=M\Paren{1-\frac{\sqrt{M^{-1}-1}}{d}}\geq M\Paren{1-\frac{\sqrt{(d^2+1)-4}}{2d}}\geq \frac{M}{2}.
\]
By Lemma~\ref{second_mean}, there exists $c\in(M-\delta M, M+\delta M)\subseteq(M/2, 3M/2)$ such that
\[
\Abs{{f_{1'}}(u)+{f_{1'}}(v)-2{f_{1'}}\Paren{\frac{u+v}{2}}}\leq \frac{1}{4}\cdot \Paren{2 \frac{1}{d} \sqrt{M(1-M)}}^2\cdot \Abs{{f_{1'}}''(c)}.
\]
By Property 4 in Section~\ref{prop_g}, 
\[
|{f_{1'}}''(c)|=|(ac_l\log n)^2 f_{1}''((ac_l\log n)\cdot c)|\leq (ac_l\log n)^2\cdot (\log 4) \cdot \frac{1}{(ac_l\log n)\cdot c}\leq (\log 8) \cdot \frac{ac_l\log n}{M}.
\]
This immediately implies
\[
\frac{1}{4}\cdot \Paren{2 \frac{1}{d} \sqrt{M(1-M)}}^2\cdot \Abs{{f_{1'}}''(c)}\leq \frac{1}{d^2} M(1-M)\cdot (\log 8) \cdot \frac{ac_l\log n}{M}\leq (\log 8)\cdot \frac{c_l \varepsilon}{c_s^2}.
\]
Consolidating all the previous results, we get
\[
\min_{g\in P_d} \max_{x\in I_n'}|{f_{1}}(x)-g(x)|\leq \mathcal{O}\Paren{\varepsilon}.
\]
Similarly, for the function ${f_{2}}$, we also have
\[
\min_{g\in P_d} \max_{x\in I_n'}|{f_{2}}(x)-g(x)|\leq \mathcal{O}\Paren{\varepsilon}.
\]
In the next section, we use these two inequalities to analyze our refined entropy estimator.
\subsection{Constructing the refined estimator}\label{bias_correction}
For our purpose, we need to approximate $B_{na-1}(h_{na},x)-h_{na}(x)$ over the interval $I_n=[0,c_l\log n/n]$ by a degree-$d$ polynomial. 
By Lemma~\ref{app_gz}, for $x\in I_n$ and $z_1:=(na-1)x\in I_n'=[0, a c_l \log n]$,
\[
h_{na}(x)-B_{na-1}(h_{na}, x)=(h(z_1+1)-{f_{2}}(z_1))-(h(z_1)-{f_{1}}(z_1))+\tilde{\mathcal{O}}\Paren{\frac{1}{n}}.
\]
By the results in~\cite{approxconst}, 
\begin{align*}
\min_{g\in P_d} \max_{x\in I_n'}|h(x)-g(x)|
&= (a c_l \log n) \min_{g\in P_d} \max_{x\in [0,1]}|h(x)-g(x)|\\
&\leq \mathcal{O}\Paren{\frac{a c_l \log n }{(c_s\log n)^2}}\\
&\leq \mathcal{O}\Paren{\varepsilon}
\end{align*}
and
\[
\min_{g\in P_d} \max_{x\in I_n'}|h(x+1)-g(x)|\leq \mathcal{O}\Paren{\varepsilon}.
\]
Combining these bounds with the last two inequalities in the last section, we get
\[
\min_{g\in P_{d-1}} \max_{x\in I_n}|(h_{na}(x)-B_{na-1}(h_{na}, x))-g(x)|\leq \mathcal{O}\Paren{\varepsilon}.
\]
Let $\tilde{g}(x)$ be the min-max polynomial that achieves the above minimum. By the derivations in Section~\ref{approx_der}, the degree-$(d-1)$ polynomial $\tilde{h}_{na}(x)$ satisfies
\[
\max_{x\in I_n}|h_{na}(x)-\tilde{h}_{na}(x)|\leq \mathcal{O}\Paren{\varepsilon}.
\]
Denote $\tilde{h}^*(x):=-\tilde{g}(x)+\tilde{h}_{na}(x)$, and note that by definition, $B_{na}'(h,x)=B_{na-1}(h_{na}, x)$. The triangle inequality implies 
\[
\max_{x\in I_n} |B_{na}'(h,x)-\tilde{h}^*(x)|=\max_{x\in I_n}|B_{na-1}(h_{na}, x)-\tilde{h}^*(x)|\leq \mathcal{O}\Paren{\varepsilon}.
\]
By a simple argument, the degree-$d$ polynomial
\[
\tilde{H}^*(x):= \int_0^{x}\tilde{h}^*(t) dt,
\]
approximates $B_{na}(h, x)$ with the following pointwise error guarantee. 
\begin{Lemma}
For any $x\in I_n$,
\[
|B_{na}(h,x)-\tilde{H}^*(x)|\leq \mathcal{O}\Paren{x\epsilon}.
\]
\end{Lemma}
In other words, $\tilde{H}^*(x)$ is a degree-$d$ polynomial that well approximates $B_{na}(h,x)$ pointwisely.

Next we argue that the coefficients of $\tilde{H}^*(x)$ can not be too large. For notational convenience, let $\tilde{h}^*(x)=\sum_{v=0}^{d-1}a_v x^v$.
By Corollary~\ref{Bern_error_h_n}, for $x\in I_n$,
\[
|h_{na}(x)-B_{na-1}(h_{na}, x)|\leq 1.
\]
Furthermore, for $x\in I_n$, $h_{na}(x)$ is an increasing function and thus
\[
|h_{na}(x)|
= \max\Brace{|h_{na}(0)|,h_{na}\Paren{\frac{c_l(\log n)}{n}}}
\leq \mathcal{O}(\log n). 
\]
Hence, over $I_n$,
\[
|\tilde{h}^*(x)|\leq \mathcal{O}(\log n).
\]
Due to the boundedness of $\tilde{h}^*(x)$, its coefficients cannot be too large:
\[
|a_v|\leq  \mathcal{O}\Paren{2^{4.5d}\log n}  \Paren{\frac{n}{c_l\log n}}^v.
\]
Write $\tilde{H}^*(x)$ as
$\tilde{H}^*(x) =  \sum_{t=1}^{d} a_t' x^t$.
Then by $\tilde{H}^*(x):= \int_0^{x}\tilde{h}^*(t) dt$ and the above bound on $|a_v|$, 
\[
|a_t'|\leq \Paren{\frac{n}{c_l\log n}}^{t-1} \mathcal{O}(2^{4.5d}).
\]
The construction of the new entropy estimator follows by replacing $\tilde{H}_{na}(x)$ with $\tilde{H}^*(x)$ in Section~\ref{Est_const}. The rest of the proof is almost the same as that in the main paper and thus is omitted.

\section{Competitive estimators for general additive properties}\label{comp_uniformity}
Consider an arbitrary real function $f:[0,1]\to \mathbb{R}$. Without loss of generality, we assume that $f(0)=0$. According to the previous derivations, we can write $B_n'(f,x)$ as
\[
B_n'(f,x):= \sum_{j=0}^{n-1}n\Paren{f\Paren{\frac{j+1}{n}}-f\Paren{\frac{j}{n}}}\binom{n-1}{j}x^{j}(1-x)^{(n-1)-j}.
\]
Our objective is to approximate $B_{na}'(f,x)$ with a low degree polynomial. For now, let us assume that $f$ is a $1$-Lipschitz function.
For $x\in [0,1]$, set $z=nx$. Denote $g_{n+1}(j):= (n+1) f\Paren{\frac{j}{n+1}}$, 
\[
f_{1,n+1}(z):=e^{-z}\sum_{j=0}^{\infty}g_{n+1}(j+1)\frac{z^j}{j!},
\]
and 
\[
f_{2,n+1}(z):=e^{-z}\sum_{j=0}^{\infty}g_{n+1}(j)\frac{z^j}{j!}
\]
The following lemma relates $f_{1,n+1}(z)$ and $f_{2,n+1}(z)$ to $B_{n+1}'(f,x)$.
\begin{Lemma}\label{app_gz1}
For $x\in[0,{\log^4 n}/{n}]$, 
\[
B_{n+1}'(f,x)=f_{1,n+1}(z)-f_{2,n+1}(z)+\tilde{\mathcal{O}}\Paren{\frac{1}{n}}.
\]
\end{Lemma}
\begin{proof}
By definition $z=nx$, hence $z\in[0, \log^4 n]$. We have
\begin{align*}
\sum_{j=0}^{n}(n+1) f\Paren{\frac{j+1}{n+1}}\binom{n}{j}x^{j}(1-x)^{n-j}
&=\sum_{j=0}^{n} g_{n+1}(j+1) \binom{n}{j} z^j  \frac{(n-z)^{n-j}}{n^n}\\
&=\Paren{1-\frac{z}{n}}^{n}\sum_{j=0}^{n} g_{n+1}(j+1) \binom{n}{j} z^j {(n-z)^{-j}}\\
&=\Paren{1-\frac{z}{n}}^{n}\sum_{j=0}^{n} g_{n+1}(j+1)  \frac{n^{\underline{j}}}{n^j} \frac{z^j}{j!} {\Paren{1-\frac{z}{n}}^{-j}}\\
&=e^{-z}\sum_{j=0}^{\infty}g_{n+1}(j+1)\frac{z^j}{j!}+\tilde{\mathcal{O}}\Paren{\frac{1}{n}}\\
&=f_{1,n+1}(z)+\tilde{\mathcal{O}}\Paren{\frac{1}{n}}.
\end{align*}
The second last equality is the most non-trivial step. To establish this equality, we need the following the three inequalities. 

\paragraph{Inequality 1:}
\begin{align*}
0\leq &\Paren{1-\frac{z}{n}}^{n}\sum_{j=\log^5 n+1}^{n} |g_{n+1}(j+1)| \frac{n^{\underline{j}}}{n^j} \frac{z^j}{j!} {\Paren{1-\frac{z}{n}}^{-j}}\\
&=\Paren{1-\frac{z}{n}}^{n}\sum_{j=\log^5 n+1}^{n} (j+1) \frac{n^{\underline{j}}}{2^j(n-z)^j} \frac{(2z)^j}{j!}\\
&\leq e^{-z}\sum_{j=\log^5 n+1}^{n} (j+1) \frac{(2z)^j}{j!}\\
&\leq e^{-z}\sum_{j=\log^5 n+1}^{n}  2 j (j-1)\frac{(2z)^j}{j!}\\
&\leq 8 z^2 e^{-z}\sum_{j=\log^5 n-1}^{n}  \frac{(2z)^j}{j!}\\
&\leq 8 (\log^8 n) \Pr(\Poi(2z)\geq \log^5 n-1)\\
&\leq \frac{1}{n}.
\end{align*}

\paragraph{Inequality 2:}
\begin{align*}
0\leq &e^{-z}\sum_{j=\log^5 n+1}^{\infty}|g_{n+1}(j+1)|\frac{z^j}{j!}\leq e^{-z}\sum_{j=\log^5 n+1}^{\infty}(j+1)\frac{z^j}{j!} \leq \frac{1}{n}. 
\end{align*}

\paragraph{Inequality 3:} For $j\leq \log^5 n$,
\begin{align*}
\Abs{ e^{-z} - \Paren{1-\frac{z}{n}}^{n} \frac{n^{\underline{j}}}{n^j}  {\Paren{1-\frac{z}{n}}^{-j}}}
&= \Abs{ e^{-z} - \Paren{1-\frac{z}{n}}^{n} \frac{n^{\underline{j}}}{(n-z)^j}}\\
&\leq  \Abs{ e^{-z} -\Paren{1-\frac{z}{n}}^{n}}+\Paren{1-\frac{z}{n}}^{n}\Abs{1- \frac{n^{\underline{j}}}{(n-z)^j}}\\
&\leq  e^{-z}\frac{z^2}{n}+e^{-z}\Abs{1- \frac{n^{\underline{j}}}{(n-z)^j}}\\
&\leq  e^{-z}\frac{z^2}{n}+e^{-z}\Paren{\Abs{1- \frac{n^{j}}{(n-z)^j}}\lor \Abs{1- \frac{(n- \log^5 n)^{j}}{(n-z)^j}}}\\
&\leq  e^{-z}\frac{z^2}{n}+e^{-z}\Paren{\Abs {\exp\Paren{\frac{zj}{n-z}}-1}\lor \Abs{\frac{(\log^5 n-z) {j}}{n-z}}}\\
&\leq  e^{-z}\frac{z^2}{n}+e^{-z}\Paren{\Abs {\frac{zj}{n-z(j+1)}}\lor \Abs{\frac{(\log^5 n) {j}}{n-z}}}\\
&\leq e^{-z}\frac{2\log^{10} n}{n}.
\end{align*}
Note that the last inequality implies 
\begin{align*}
&\Abs{e^{-z}\sum_{j=0}^{\log^5 n}\frac{z^j}{j!}g_{n+1}(j+1)-\Paren{1-\frac{z}{n}}^{n}\sum_{j=0}^{\log^5 n} g_{n+1}(j+1) \frac{n^{\underline{j}}}{n^j} \frac{z^j}{j!} {\Paren{1-\frac{z}{n}}^{-j}}}\\
&\leq \frac{2\log^{10} n}{n}\cdot e^{-z}\sum_{j=0}^{\log^5 n}\frac{z^j}{j!} (j+1)\\
&\leq  \frac{2\log^{10} n}{n}\cdot (1+2z)\\
&\leq \frac{5\log^{14} n}{n}.
\end{align*}
This together with Inequality 1 and 2 proves the desired equality. 
Similarly, we have
\[
\sum_{j=0}^{n}(n+1) f\Paren{\frac{j}{n+1}}\binom{n}{j}x^{j}(1-x)^{n-j}
=f_{2,n+1}(z)+\tilde{\mathcal{O}}\Paren{\frac{1}{n}}.
\]
This completes the proof.
\end{proof}
Re-define $z:=(na-1) x$. Lemma~\ref{app_gz1} immediately implies that for $x\in I_n=[0,c_l(\log n)/n]\subseteq[0,{(\log^4 (na-1))}/{(na-1)}]$, 
\[
B_{na}'(f,x)=f_{1,na}(z)-f_{2,na}(z)+\tilde{\mathcal{O}}\Paren{\frac{1}{na}}.
\]
Note that in this case $z\in I_n' =[0, a c_l \log n]$. Let $t_{na}(z):= f_{1,na}(z)-f_{2,na}(z)$ and $r_{na}(j):=g_{na}(j+2)+g_{na}(j)-2g_{na}(j+1)$. Then direct calculation yields,
\begin{align*}
t''_{na}(z)
&=e^{-z}\sum_{j=0}^{\infty}r_{na}(j+1)\frac{z^j}{j!}-e^{-z}\sum_{j=0}^{\infty}r_{na}(j)\frac{z^j}{j!}\\
&=e^{-z}\sum_{j=0}^{\infty}r_{na}(j+1)\frac{z^j}{j!}-e^{-z}r_{na}(0)-\sum_{j=0}^{\infty}r_{na}(j+1)\frac{z^{j+1}}{(j+1)!}\\
&=e^{-z}\sum_{j=0}^{\infty}r_{na}(j+1)\Paren{\frac{z^j}{j!}-\frac{z^{j+1}}{(j+1)!}}-e^{-z}r_{na}(0)
\end{align*}
Since we assume that $f$ is $1$-Lipschitz, $|r_{na}(j)|\leq 2$. Therefore, for $z\in I_n'$,
\begin{align*}
|t''_{na}(z)|
&\leq e^{-z}\sum_{j=0}^{\infty}|r_{na}(j+1)|\Paren{\frac{z^j}{j!}+\frac{z^{j+1}}{(j+1)!}}+e^{-z}|r_{na}(0)|\leq 6.
\end{align*}
We can bound each individual term by the following lemma.
\begin{Lemma}\label{boundl}
For $j\geq 1$ and $z\geq 0$, we have
\[
\Abs{e^{-z} \Paren{\frac{z^{j}}{j!}-\frac{z^{j+1}}{(j+1)!}}}\leq \frac{1}{\sqrt{2\pi}((j+1)-\sqrt{j+1})}
\]
and
\[
\Abs{e^{-z} \Paren{\frac{z^{j}}{j!}-\frac{z^{j+1}}{(j+1)!}}}\leq \frac{5}{z}.
\]
\end{Lemma}
\begin{proof}
Let us denote 
\[
q_1(z):=e^{-z} \Paren{\frac{z^{j}}{j!}-\frac{z^{j+1}}{(j+1)!}}.
\]
The derivative of $q_1(z)$ is 
\begin{align*}
q'_1(z)
& = -e^{-z} \frac{z^{j}}{j!} +e^{-z} \frac{z^{j-1}}{(j-1)!}+e^{-z}\frac{z^{j+1}}{(j+1)!}-e^{-z} \frac{z^{j}}{j!}\\
& = e^{-z} \frac{z^{j-1}}{(j+1)!}\Paren{-2(j+1)z+j(j+1)+z^2}.
\end{align*}
Set $q'_1(z)=0$ and note that $q_1(0)=\lim_{z\to \infty} q_1(z) =0$, the maximum of $|q_1(z)|$ is attained at $z_1:=(j+1)-\sqrt{j+1}$ or $z_2:=(j+1)+\sqrt{j+1}$. We consider $z_1$ first.
\begin{align*}
|q_1(z_1)|
&=e^{-z_1} \frac{z_1^{j+1}}{(j+1)!}\Abs{\frac{j+1}{z_1}-1}\\
&\leq e^{-(j+1)+\sqrt{j+1}} ((j+1)-\sqrt{j+1})^{j+1}\frac{e^{j+1}}{\sqrt{2\pi}(j+1)^{j+1+1/2}}\frac{1}{\sqrt{j+1}-1}\\
&\leq e^{\sqrt{j+1}} \Paren{1-\frac{1}{\sqrt{j+1}}}^{j+1}\frac{1}{\sqrt{2\pi}\sqrt{j+1}}\frac{1}{\sqrt{j+1}-1}\\
&\leq \frac{1}{\sqrt{2\pi}((j+1)-\sqrt{j+1})}.
\end{align*}
Similarly, for $z_2$, we also have $|q_1(z_1)|\leq {1}/{(\sqrt{2\pi}((j+1)+\sqrt{j+1}))}$.
Analogously, let us denote 
\[
q_2(z):=e^{-z} \Paren{\frac{z^{j+1}}{j!}-\frac{z^{j+2}}{(j+1)!}}.
\]
The derivative of $q_2(z)$ is 
\begin{align*}
q'_2(z)
& = e^{-z} \frac{z^j}{(j+1)!}\Paren{-(2j+3)z+(j+1)^2+z^2}.
\end{align*}
Set $q'_2(z)=0$ and note that $q_2(0)=\lim_{z\to \infty} q_2(z) =0$, the maximum of $|q_2(z)|$ is attained at $z_3:=((2j+3)-\sqrt{4j+5})/2$ or $z_4:=((2j+3)+\sqrt{4j+5})/2$. Note that $|z_3|, |z_4|\leq 2(j+2)$. Therefore,
\begin{align*}
|q_2(z_3)|= |z_3| |q_1(z_3)|\leq 2(j+2) \max_{z}|q_1(z)|
\leq \frac{2(j+2)}{\sqrt{2\pi}((j+1)-\sqrt{j+1})}\leq 5, \forall j\geq 1.
\end{align*}
The same proof also shows that $|q_2(z_4)|\leq 5$.
\end{proof}
\subsection{$\ell_1$-distance}
Now let us focus on the problem of estimating the $\ell_1$-distance between the unknown distribution $\vec{p}\in \Delta_k$ and a given distribution $\vec{q}\in \Delta_k$. Since our estimator is constructed symbol by symbol, it is sufficient to consider the problem of approximating $\ell_q(x)=|x-q|-q$. 

Set $g_{n+1}(j):= (n+1) \ell_q\Paren{\frac{j}{n+1}}$. We note that $r_{na}(j)$ equals $0$ for all but at most two different values of $j$. Therefore, by Lemma~\ref{boundl}, for all $z\in I_n'$,
we have $|t''_{na}(z)|\leq \mathcal{O}(1)$, and
 $|t''_{na}(z)|\leq \mathcal{O}(1)z^{-1}$,
 where the first and second inequalities resemble Property 3 and 4 in Section~\ref{prop_g}, respectively. Using arguments similar to those in Section~\ref{app_f1} and~\ref{bias_correction}, we can construct an estimator for $D_{\vec{q}}(\vec{p})$ that provides the guarantees stated in Theorem~\ref{thm2.1}. Note that concavity/convexity is actually not crucial to establishing the final result in Section~\ref{app_f1}. Also note that we need to replace our analysis in Section~\ref{Bias_large} and~\ref{Var_large} for the corresponding large-probability estimator by that in~\cite{nips}.

\subsection{General additive properties}
More generally, the results on $\ell_1$-distance hold for any additive property $F(\vec{p})=\sum_{i\in[k]} f_i(p_i)$ that satisfies the simple condition: $f_i$ is $\mathcal{O}(1)$-Lipschitz, for all $i$. Without loss of generality, assume that all $f_i$'s are $1$-Lipschitz and satisfy $f_i(0)=0$. By the previous derivations, we immediately have $|t''_{na}(z)|\leq 6$, which recovers Property 3 in Section~\ref{app_f1}. Again, concavity/convexity is actually not necessary to establishing the final result in Section~\ref{app_f1}. The proof will be complete if we also recover Property 4 in that section. In other words, we only need to show: $|t''_{na}(z)z|\leq \mathcal{O}(1)$, where
\[
t''_{na}(z)z=e^{-z}\sum_{j=0}^{\infty}r_{na}(j+1)\Paren{\frac{z^{j+1}}{j!}-\frac{z^{j+2}}{(j+1)!}}-e^{-z}zr_{na}(0).
\]
Fix $z\in I_n'$ and treat it as a constant. Let $b_j:=r_{na}(j+1)$ and $a_j:=e^{-z}\Paren{\frac{z^{j+1}}{j!}-\frac{z^{j+2}}{(j+1)!}}$. By Lemma~\ref{boundl}, we have $|a_j|\leq 5, \forall j\geq 1$. Note that there is need to worry about 
the slack term $e^{-z}zr_{na}(0)$ and the first term in the sum which corresponds to $j=0$, since the absolute values of both terms contribute at most $\mathcal{O}(1)$ to the expression for any $z\geq 0$. The key observation is that any consecutive partial sum of sequence $\{b_j\}_{j\geq 1}$ is also bounded by $\mathcal{O}(1)$ in magnitude. Specifically, for any $n_1, n_2\in \mathbb{Z}^+$ satisfying $n_1+2\leq n_2$,
\begin{align*}
\Abs{\sum_{j=n_1}^{n_2} b_j}
&=\Abs{\sum_{j=n_1}^{n_2} r_{na}(j+1)}\\
&=\Abs{\sum_{j=n_1}^{n_2} (g_{na}(j+3)+g_{na}(j+1)-2g_{na}(j+2))}\\
&=\Abs{\sum_{j=n_1+3}^{n_2+3} g_{na}(j)+\sum_{j=n_1+1}^{n_2+1}g_{na}(j)-2\sum_{j=n_1+2}^{n_2+2}g_{na}(j)}\\
&=\Abs{(g_{na}(n_2+3)-g_{na}(n_2+2))+(g_{na}(n_1+1)-g_{na}(n_1+2))}\\
&\leq 2.
\end{align*}
Furthermore, the sequence $\{a_j\}_{j\geq 1}$ can change its monotonicity at most two times. We can prove this claim by considering the sign of $a_j-a_{j-1}$. More concretely,
\begin{align*}
\sign\Paren{a_j-a_{j-1}}
&= \sign\Paren{e^{-z}\Paren{\frac{z^{j+1}}{j!}-\frac{z^{j+2}}{(j+1)!}}-e^{-z}\Paren{\frac{z^{j}}{(j-1)!}-\frac{z^{j+1}}{j!}}}\\
&= \sign\Paren{2(j+1)z-z^2-(j+1)j}\\
&= \sign\Paren{-j^2+j(2z-1)+(2z-z^2)}.
\end{align*}
Since $z$ is fixed, the last expression can change its value at most two times as $j$ increases from $0$ to infinity. The last piece of the proof is the following corollary of the well-known Abel's inequality.
\begin{Lemma}\label{abel_ineq}
Let $\{a'_j\}_{j=1}^m$ be a sequence of real numbers that is either increasing or decreasing, and let $\{b'_j\}_{j=1}^m$ be a sequence of real or complex numbers. Then,
\[
|\sum_{j=1}^m a'_j b'_j|\leq \max_{t=1,\ldots, m} |B'_t| (2|a'_n|+|a'_1|),
\]
where $B'_t:=\sum_{j=1}^t b'_t$.
\end{Lemma}
By the previous discussion, we can find two indices $j_1$ and $j_2$, such that $\{a_j\}_{j=1}^{j_1}$, $\{a_j\}_{j=j_1+1}^{j_2}$, and $\{a_j\}_{j\geq j_2+1}$ are all monotone subsequences. Then, we apply Lemma~\ref{abel_ineq} to each of them and further bound the resulting quantities by the two inequalities proved above: $\Abs{\sum_{j=n_1}^{n_2} b_j}\leq \mathcal{O}(1)$ and $|a_j|\leq 6,\forall j\geq 1$.  This concludes the proof. Finally, we would like to point out that the above argument actually applies to a much broader class of additive properties beyond the Lipschitz one, which we will not address here for the sake of clarity and simplicity.

\section{A competitive estimator for support size}\label{comp_size}
\subsection{Estimator construction}
Recall that
\[
s(x)=\indic_{x>0}.
\]
Let $\vec{p}$ and $S_{\vec{p}}$ denote an unknown distribution and its support size.
Re-define $a:=|\log^{-2} \epsilon|\cdot\log S_{\vec{p}}$. Let $X^{na}$ be a sample sequence drawn from $\vec{p}$, and $N_i''$ be the number of times symbol $i$ appears. 

The $na$-sample empirical estimator estimates the support size by
\[
\hat{S}^E(X^{na}):=\sum_{i\in[k]} \indic_{N_i''>0}.
\]
Taking expectation, we have
\[
\EE[\hat{S}^E(X^{na})]:=\sum_{i\in[k]} \EE[\indic_{N_i''>0}] = \sum_{i\in [k]} (1-(1-p_i)^{na}).
\]
Following~\cite{mmcover, pnas}, having a length-$\Poi(n)$ sample $X^N$, we denote by $\phi_j$ the number of symbols that appear $j$ times and estimate $\EE[\hat{S}^E(X^{na})]$ by
\[
\hat{S}(X^N):=\sum_{j=1}^{\infty} \phi_j (1-(-(a-1))^j \Pr(Z\geq j)),
\]
where $Z\sim \Poi(r)$ for some parameter $r$. In addition, we define $N_i$ as the number of times symbol $i$ appears. By the property of Poisson sampling, all the $N_i$'s are independent. 

\subsection{Bounding the bias}
The following lemma bounds the bias of $\hat{S}(X^N)$ in estimating $\EE[\hat{S}^E(X^{na})]$.
\begin{Lemma}\label{supp_bias}
For all $a\geq 1$,
\[
{|\EE[\hat{S}(X^N)] - \EE[\hat{S}^E (X^{na})]|}\leq \min\Brace{na, S_{\vec{p}}} e^{-r}+2.
\]
\end{Lemma}
\begin{proof}
Noting that for any $m\geq 0$ and $p\in[0,1]$,
\[
0\leq e^{-mp}-(1-p)^m\leq 2p,
\]
we have
\begin{align*}
&{|\EE[\hat{S}(X^N)] - \EE[\hat{S}^E(X^{na})]|}\\
&= \Abs{\EE\Br{\sum_{j}^{\infty}\phi_j} -\EE\Br{\sum_{j}^{\infty}\phi_j (-(a-1))^j \Pr(Z\geq j)}-\sum_{i\in [k]} (1-(1-p_i)^{na})}\\
&= \Abs{\sum_{i\in [k]} (1-e^{-np_i})-\EE\Br{\sum_{j}^{\infty}\phi_j (-(a-1))^j \Pr(Z\geq j)}-\sum_{i\in [k]} (1-(1-p_i)^{na})}\\
&\leq \Abs{\sum_{i\in [k]} (-e^{-np_i})-\EE\Br{\sum_{j}^{\infty}\phi_j (-(a-1))^j \Pr(Z\geq j)}-\sum_{i\in [k]} (-e^{-na p_i})}+2\sum_{i\in[k]}p_i \\
&=\Abs{\sum_{i\in [k]} e^{-np_i}(e^{-n(a-1) p_i}-1)-\EE\Br{\sum_{j}^{\infty}\phi_j (-(a-1))^j \Pr(Z\geq j)}}+2 \\
&\leq \min\Brace{na, S_{\vec{p}}} e^{-r}+2,
\end{align*}
where the last step follows by Lemma 7 and Corollary 2 in~\cite{pnas}.
\end{proof}

\subsection{Bounding the mean absolute deviation}
\subsubsection{Bounds for $\boldsymbol{\hat{S}(X^N)}$}
In this section, we analyze the mean absolute deviation of $\hat{S}(X^N)$.
To do this, we need the following two lemmas.
The first lemma bounds the coefficients of this estimator.
\begin{Lemma}~\cite{mmcover}
For $j\geq 1$ and $a\geq 1$,
\[
|1-(-(a-1))^j \Pr(Z\geq j)|\leq 1+e^{r(a-1)}.
\]
\end{Lemma}
The second lemma is the McDiarmid's inequality.
\begin{Lemma}\label{McD_ineq}
Let ${Y_1,\ldots,Y_m}$ be independent random variables taking values in ranges ${R_1,\ldots,R_m}$, 
and let $F:R_1\times\ldots\times R_m\rightarrow C$  with the property that if one freezes all but the ${w^{th}}$ coordinate 
of ${F(y_1,\ldots,y_m)}$ for some ${1 \leq w \leq m}$, then ${F}$ only fluctuates by most ${c_w > 0}$, thus
$ |F(y_1,\ldots,y_{w-1},y_w,\\y_{w+1},\ldots,y_m) - F(y_1,\ldots,y_{w-1},y_w',y_{w+1},\ldots,y_m)| \leq c_w$ 
for all ${y_j \in R_j}$ and ${y'_w \in R_w}$ for ${1 \leq j \leq m}$. Then for any ${\lambda > 0}$, one has
$\displaystyle \Pr( |F(Y) - \EE[F(Y)]| \geq \lambda \sigma ) \leq C \exp(- c \lambda^2 )$
for some absolute constants ${C,c>0}$, where ${\sigma^2 := \sum_{j=1}^m c_j^2}$.
\end{Lemma}
Note that $\hat{S}(X^N)$, when viewed as a function of $N_i$'s with indexes $i$ satisfying $p_i\not=0$, fullfills the property described in Lemma~\ref{McD_ineq}, with $m=S_{\vec{p}}$ and $c_w=2+2e^{r(a-1)}$ for all $1\leq w\leq m$. Therefore, for $\sigma^2 := 4S_{\vec{p}}(1+e^{r(a-1)})^2$,
\[
\Pr( |\hat{S}(X^N) - \EE[\hat{S}(X^N)]| \geq \lambda \sigma ) \leq C \exp(- c \lambda^2 ).
\]
This further implies
\begin{align*}
\EE\Abs{\hat{S}(X^N) - \EE[\hat{S}(X^N)]}
&=\int_{0}^{\infty} \Pr( |\hat{S}(X^N) - \EE[\hat{S}(X^N)]| \geq t)\ dt\\
&=\sigma \int_{0}^{\infty} \Pr( |\hat{S}(X^N) - \EE[\hat{S}(X^N)]| \geq \lambda \sigma)\ d\lambda \\
&\leq C\sigma \int_{0}^{\infty} \exp(- c \lambda^2 )d\lambda \\
&\leq \mathcal{O}(\sqrt{S_{\vec{p}}}(1+e^{r(a-1)})).
\end{align*}
Analogously, viewing $\hat{S}(X^N)$ as a function of $X_i$'s implies 
\[
\EE\Abs{\hat{S}(X^N) - \EE[\hat{S}(X^N)]}\leq \mathcal{O}(\sqrt{n}(1+e^{r(a-1)})).
\]
Hence,
\[
\EE\Abs{\hat{S}(X^N) - \EE[\hat{S}(X^N)]}\leq \mathcal{O}\Paren{\sqrt{\min\Brace{S_{\vec{p}},n}}(1+e^{r(a-1)})}.
\]

\subsubsection{Bounds for $\boldsymbol{\hat{S}^E(X^{na})}$}
The following lemma bounds the variance of $\hat{S}^E(X^{na})$ in terms of $S_{\vec{p}}$.
\begin{Lemma}
For $m\geq 1$ and $X^{m}\sim \vec{p}$,
\[
\Var(\hat{S}^E(X^{m}))\leq \mathcal{O}(S_{\vec{p}}).
\]
\end{Lemma}
\begin{proof}
Let $N_i$ denote the number of times symbol $i$ appears in $X^m$. By independence, 
\begin{align*}
\Var(\hat{S}^E(X^{m}))
&=\Var\Paren{\sum_{i: p_i>0}\indic_{N_i>0}}\\
&=\EE\Br{\Paren{\sum_{i: p_i>0}\indic_{N_i>0}}^2}-\Paren{\EE\Br{\sum_{i: p_i>0}\indic_{N_i>0}}}^2
\end{align*}
Let $M\sim\Poi(m)$ and $X^M$ be an independent sample of length $M$. Let $N_i'$ denote the number of 
times symbol $i$ appears in $X^M$.
We have
\begin{align*}
&\EE\Br{\Paren{\sum_{i: p_i>0}\indic_{N_i>0}}^2}\\
&=\EE\Br{\sum_{i: p_i>0}\indic_{N_i>0}+\sum_{i\not=j: p_i>0,p_j>0}\indic_{N_i>0}\indic_{N_j>0}}\\
&=\sum_{i: p_i>0}(1-\EE[\indic_{N_i=0}])+\sum_{i\not=j: p_i>0,p_j>0}\EE[(1-\indic_{N_i=0})(1-\indic_{N_j=0})]\\
&=\sum_{i: p_i>0}(1-(1-p_i)^m)+\sum_{i\not=j: p_i>0,p_j>0}\Paren{1-(1-p_i)^m-(1-p_j)^m+(1-p_i-p_j)^m}
\end{align*}
Noting that for any $m\geq 0$ and $p\in[0,1]$,
\[
0\leq e^{-mp}-(1-p)^m\leq 2p,
\]
we have
\[
|(1-(1-p_i)^m) - (1-e^{-mp_i})|\leq 2p_i
\]
and
\[
|(1-(1-p_i)^m-(1-p_j)^m+(1-p_i-p_j)^m)-(1-e^{-mp_i}-e^{-mp_j}+e^{-m(p_i+p_j)})|\leq 4(p_i+p_j).
\]
Therefore, 
\begin{align*}
\Abs{\EE\Br{\Paren{\sum_{i: p_i>0}\indic_{N_i>0}}^2}-\EE\Br{\Paren{\sum_{i: p_i>0}\indic_{N_i'>0}}^2}}
&\leq \sum_{i: p_i>0}2p_i+\sum_{i\not=j: p_i>0,p_j>0}4(p_i+p_j)\\
&\leq 4\sum_{i: p_i>0}\sum_{j: p_j>0}(p_i+p_j)\\
&\leq 8S_{\vec{p}}.
\end{align*}
Similarly, 
\begin{align*}
&\Abs{\Paren{\EE\Br{\sum_{i: p_i>0}\indic_{N_i>0}}}^2-\Paren{\EE\Br{\sum_{i: p_i>0}\indic_{N_i'>0}}}^2}\\
&=\Abs{\EE\Br{\sum_{i: p_i>0}\indic_{N_i>0}}-\EE\Br{\sum_{i: p_i>0}\indic_{N_i'>0}}}\Abs{\EE\Br{\sum_{i: p_i>0}\indic_{N_i>0}}+\EE\Br{\sum_{i: p_i>0}\indic_{N_i'>0}}}\\
&\leq \Abs{\sum_{i: p_i>0}\EE\Br{\indic_{N_i>0}}-\sum_{i: p_i>0}\EE\Br{\indic_{N_i'>0}}}\cdot 2 S_{\vec{p}}\\
&\leq (\sum_{i: p_i>0}2p_i)\cdot 2 S_{\vec{p}}\\
&\leq 4  S_{\vec{p}}.
\end{align*}
Note that changing the value of a particular $N_i'$ changes the value of $\sum_{i: p_i>0}\indic_{N_i'>0}$ by at most $1$. Again, by the McDiarmid's inequality, 
\[
\Var\Paren{\sum_{i: p_i>0}\indic_{N_i'>0}}\leq \mathcal{O}(S_{\vec{p}}).
\]
The triangle inequality combines all the above results and yields
\[
\Var\Paren{\sum_{i: p_i>0}\indic_{N_i>0}}\leq \mathcal{O}(S_{\vec{p}}).
\]
\end{proof}
By Jensen's inequality, the above lemma implies
\[
\EE\Abs{\hat{S}^E(X^{na})-\EE[\hat{S}^E(X^{na})]}\leq \sqrt{\Var(\hat{S}^E(X^{na}))}\leq \mathcal{O}(\sqrt{S_{\vec{p}}}).
\]

\subsection{Proving Theorem 4}
Setting $r=|\log \epsilon|$, we get
\[
e^{r(a-1)}\leq {S_{\vec{p}}^{|\log^{-1} \epsilon|}}
\]
and
\[
e^{-r} = e^{-|\log \epsilon|}=\epsilon.
\]
Hence, by the previous results,
\begin{align*}
\EE\Abs{\hat{S}(X^N) - \hat{S}^E(X^{na})}
&\leq \EE\Abs{\hat{S}(X^N) - \EE[\hat{S}^E(X^{na})]}+\EE\Abs{\EE[\hat{S}^E(X^{na})] - \hat{S}^E(X^{na})}\\
&\leq \mathcal{O}\Paren{S_{\vec{p}}^{|\log^{-1} \epsilon|+\frac{1}{2}}+ S_{\vec{p}}\cdot \varepsilon}.
\end{align*}
Normalize both sides by $S_{\vec{p}}$. Then,
\[
\EE\Br{\Abs{\frac{\hat{S}(X^N)}{S_{\vec{p}}} - \frac{\hat{S}^E(X^{na})}{S_{\vec{p}}}}}
\leq \mathcal{O}\Paren{{S_{\vec{p}}^{|\log^{-1} \epsilon|-\frac12}}+\varepsilon}.
\]
\section{A competitive estimator for support coverage}\label{comp_coverage}
\subsection{Estimator construction}
Recall that 
\[
c(p):=1-(1-p_i)^{m},
\]
where $m$ is a given parameter.
Re-define the amplification parameter as $a:=|\log^{-2} \epsilon|\cdot\log C_{\vec{p}}$. 
Similar to the last section, let $X^{na}$ be an independent length-$na$ sample sequence drawn from $\vec{p}$, and $N_i''$ be the number of times symbol $i$ appears. 

The $na$-sample empirical estimator estimates 
$
C_{\vec{p}}=\sum_{i\in[k]} c(p_i)
$
by the quantity
\[
\hat{C}^E(X^{na}):=\sum_{i\in[k]} c(N_i''/(na))=\sum_{i\in[k]}\Paren{1-\Paren{1-\frac{N_i''}{na}}^m}.
\]
Taking expectation, we get
\[
\EE[\hat{C}^E(X^{na})]=\sum_{i\in[k]}\EE\Br{1-\Paren{1-\frac{N_i''}{na}}^m}.
\]
Let us denote
\[
T(\vec{p}):=\sum_{i\in[k]}\EE\Br{1-e^{-m\frac{N_i''}{na}}}.
\]
Noting that for $t\geq 1$ and $p\in[0,1]$,
\[
|e^{-tp}-(1-p)^t|\leq 2p,
\]
we have
\[
|\EE[\hat{C}^E(X^{na})]-T(\vec{p})]\leq \sum_{i\in[k]}\EE\Br{2\cdot \frac{N_i''}{na}}=2.
\]
Thus, it suffices to estimate $T(\vec{p})$, which satisfies
\begin{align*}
T(\vec{p})
&=\sum_{i\in[k]}\Paren{1-\EE\Br{e^{-m\frac{N_i''}{na}}}}\\
&=\sum_{i\in[k]}\Paren{1-\sum_{j=0}^{na}\binom{na}{j} p_i^j (1-p_i)^{na-j}e^{-m\frac{j}{na}}}\\
&=\sum_{i\in[k]}\Paren{1-\sum_{j=0}^{na}\binom{na}{j} \Paren{p_i\cdot e^{-\frac{m}{na}}}^j (1-p_i)^{na-j}}\\
&=\sum_{i\in[k]}\Paren{1-\Paren{1-p_i(1-e^{-\frac{m}{na}})}^{na}}.
\end{align*}
Let us denote
\[
T_1(\vec{p}):= \sum_{i\in[k]}\Paren{1-\exp\Paren{-na(1-e^{-\frac{m}{na}})p_i}}.
\]
Since $(1-e^{-\frac{m}{na}})p_i\in[0,1]$, 
we have 
\[
|T(\vec{p})-T_1(\vec{p})|\leq \sum_{i\in[k]} 2(1-e^{-\frac{m}{na}})p_i\leq 2.
\]
Define a new amplification parameter $a':=a(1-e^{-\frac{m}{na}})$. We can write $T_1(\vec{p})$ as
\[
T_1(\vec{p}):= \sum_{i\in[k]}\Paren{1-\exp\Paren{-na'p_i}}.
\]
For simplicity, we assume that $m\geq 1.5 n$ and $a>1.8$. Then 
\[
a'=a(1-e^{-\frac{m}{na}})\geq a(1-e^{-\frac{1.5}{a}})>1.
\]
Analogous to case of support size estimation, we can draw a length-$\Poi(n)$ sample sequence $X^N$ and estimate $\EE[\hat{C}^E(X^{na})]$ by
\[
\hat{C}(X^N):=\sum_{j=1}^{\infty} \phi_j (1-(-(a'-1))^j \Pr(\Poi(r)\geq j)).
\]

\subsection{Bounding the bias}
We bound the bias of $\hat{C}(X^N)$ in estimating $\EE[\hat{C}^E(X^{na})]$ as follows.
\begin{align*}
|\EE[\hat{C}(X^N)] - \EE[\hat{C}^E(X^{na})]|
&\leq |\EE[\hat{C}(X^N)] - T_1(\vec{p})|+|T_1(\vec{p})-\EE[\hat{C}^E(X^{na})]|\\
&\leq |\EE[\hat{C}(X^N)] - T_1(\vec{p})|+4\\
&= \left|\sum_{i\in[k]}e^{-np_i}(e^{-n(a'-1)p_i}-1)\right.\\
&\left.-\sum_{i\in[k]} e^{-np_i}\sum_{j=1}^{\infty} \frac{(-(a'-1)np_i)^j}{j!} \Pr(\Poi(r)\geq j)\right|+4\\
&\leq\left|\sum_{i\in[k]}e^{-np_i}\Paren{\sum_{j=1}^{\infty} \frac{(-(a'-1)np_i)^j}{j!} \Pr(\Poi(r)< j)}\right|+4.
\end{align*}
To bound the last sum, we need the following lemma.
\begin{Lemma}
For all $y, r\geq 0$,
\[
\Abs{\sum_{j=1}^{\infty} \frac{(-y)^j}{j!} \Pr(\Poi(r)< j)}\leq e^{-r} (1-e^{-y}).
\]
\end{Lemma}
\begin{proof}
By Lemma 6 of~\cite{pnas},
\begin{align*}
\Abs{\sum_{j=1}^{\infty} \frac{(-y)^j}{j!} \Pr(\Poi(r)< j)}
&\leq \max_{s\leq y}\Abs{\EE_{L\sim \Poi(r)}\Br{\frac{(-s)^L}{L!}}}(1-e^{-y})\\
&= \max_{s\leq y} \Abs{J_0(2\sqrt{sr})} e^{-r}(1-e^{-y})\\
&\leq e^{-r}(1-e^{-y}),
\end{align*}
where $J_0$ is the Bessel function of the first kind and satisfies $|J_0(x)|\leq 1,\forall x\geq 0$~\cite{mi64}.
\end{proof}
By the above lemma, we have
\begin{align*}
|\EE[\hat{C}(X^N)] - \EE[\hat{C}^E(X^{na})]|
&\leq\left|\sum_{i\in[k]}e^{-np_i}\Paren{\sum_{j=1}^{\infty} \frac{(-(a'-1)np_i)^j}{j!} \Pr(\Poi(r)< j)}\right|+4\\
&\leq e^{-r}\sum_{i\in[k]}e^{-np_i}(1-e^{-(a'-1)np_i})+4\\
&\leq e^{-r}\sum_{i\in[k]}(1-e^{-na'p_i})+4.
\end{align*}
Note that $na'=na(1-e^{-\frac{m}{na}})\leq m$. Hence,
\[
{|\EE[\hat{C}(X^N)] - \EE[\hat{C}^E(X^{na})]|}\leq e^{-r}\sum_{i\in[k]} (1-e^{-m p_i}) +4=e^{-r}C_{\vec{p}} +4.
\]

\subsection{Bounding the mean absolute deviation}
\subsubsection{Bounds for $\boldsymbol{\hat{C}(X^N)}$}
Now we bound the mean absolute deviation of $\hat{C}(X^N)$ in terms of $C_{\vec{p}}$. 
By the Jensen's inequality,
\begin{align*}
\EE\Abs{\hat{C}(X^N) - \EE[\hat{C}(X^N)]}
&\leq \sqrt{\Var\Paren{\hat{C}(X^N)}}\\
&=\sqrt{\sum_{i\in k}\Var\Paren{{\sum_{j=1}^{\infty}\indic_{N_i=j}  (1-(-(a'-1))^j \Pr(\Poi(r)\geq j))}}}\\
&\leq \sqrt{\sum_{i\in k}\EE\Br{\Paren{\sum_{j=1}^{\infty}\indic_{N_i=j}  (1-(-(a'-1))^j \Pr(\Poi(r)\geq j))}^2}}\\
&=\sqrt{\sum_{i\in k}\sum_{j=1}^{\infty}\EE\Br{\indic_{N_i=j}} \Paren{ 1-(-(a'-1))^j \Pr(\Poi(r)\geq j)}^2}\\
&\leq (1+e^{r(a'-1)})\sqrt{\sum_{i\in k}(1-e^{-np_i})}
\end{align*}

By our assumption that $m\geq 1.5 n$, 
\begin{align*}
\EE[|\hat{C}(X^N) - \EE[\hat{C}(X^N)]|]
&\leq (1+e^{r(a'-1)})\sqrt{\sum_{i\in k}(1-e^{-np_i})}\\
&\leq (1+e^{r(a'-1)})\sqrt{\sum_{i\in k}(1-e^{-mp_i})}\\
&\leq (1+e^{r(a'-1)})\sqrt{\sum_{i\in k}(1-(1-p_i)^m)}\\
&= (1+e^{r(a'-1)})\sqrt{C_{\vec{p}}}.
\end{align*}
\subsubsection{Bounds for $\boldsymbol{\hat{C}^E(X^{na})}$}
It remains to bound the mean absolute deviation of the $na$-sample empirical estimator. To deal with the dependence among the counts $N_i''$'s, we need the following definition and lemma~\cite{jk83}.
\begin{definition}
Random variables $X_1,\ldots,X_S$ are said to be \emph{negatively associated} if
for any pair of disjoint subsets $A_1,A_2$ of ${1, 2, \ldots, S}$, and
any component-wise increasing functions $f_1, f_2$,
\[
\text{Cov}(f_1(X_i, i\in A_1), f_2(X_j, j\in A_2))\leq 0.
\]
\end{definition}
Next lemma can be used to check whether random variables are negatively associated or not.
\begin{Lemma}\label{neg_associate}
Let $X_1,\ldots,X_S$ be $S$ independent
random variables with log-concave densities. Then
the joint conditional distribution of $X_1,\ldots,X_S$ given 
$\sum_{i=1}^S X_i$ is negatively associated.
\end{Lemma}
By Lemma~\ref{neg_associate}, $N_i''$'s are negatively correlated.
Furthermore, note that 
\[
c^*(x):=1-\Paren{1-\frac{x}{na}}^m
\]
is an increasing function, and
\[
\hat{C}^E(X^{na}):=\sum_{i\in[k]} c^*(N_i'').
\]
Hence for any $i,j\in[k]$ such that $i\not=j$,
\[
\text{Cov}(c^*(N_i''),c^*(N_j''))\leq 0.
\]
Therefore,
\begin{align*}
\Var\Paren{\hat{C}^E(X^{na})}
&=\sum_{i\in[k]} \Var(c^*(N_i''))+2\sum_{i,j\in[k], i\not=j}\text{Cov}(c^*(N_i''),c^*(N_j''))\\
&\leq \sum_{i\in[k]} \Var(c^*(N_i''))\\
&\leq \sum_{i\in[k]} \EE[(c^*(N_i''))^2]\\
&=\sum_{i\in[k]} \EE\Br{\sum_{j=0}^{na}\indic_{N_i=j} (C^*(j))^2}\\
&\leq \sum_{i\in[k]} \sum_{j=1}^{na}\EE\Br{\indic_{N_i=j}} \\
&=\sum_{i\in[k]} (1-(1-p_i)^{na}).
\end{align*}
Without loss of generality, we can assume that $a$ is a positive integer. Then,
\begin{align*}
\sum_{i\in[k]} (1-(1-p_i)^{na})
&=\sum_{i\in[k]} (1-(1-p_i)^{n})(\sum_{j=0}^{a-1}(1-p_i)^{nj})\\
&\leq a\sum_{i\in[k]} (1-(1-p_i)^{n})\\
&\leq a \sum_{i\in[k]} (1-(1-p_i)^{m})\\
&=aC_{\vec{p}}.
\end{align*}
The Jensen's inequality implies that
\[
\EE\Abs{\hat{C}^E(X^{na})-\EE[\hat{C}^E(X^{na})]}\leq \sqrt{\Var(\hat{C}^E(X^{na}))}\leq \sqrt{a C_{\vec{p}}}.
\]

\subsection{Proving Theorem 5}
The triangle inequality consolidates the major inequalities above and yields
\[
\EE\Abs{\hat{C}(X^N)-\hat{C}^E(X^{na})}\leq \mathcal{O}\Paren{e^{-r}C_{\vec{p}} +4+\sqrt{a C_{\vec{p}}}+ (1+e^{r(a'-1)})\sqrt{C_{\vec{p}}}}.
\]
Using the fact that $a'<a=|\log^{-2} \epsilon|\cdot\log C_{\vec{p}}$ and set $r=|\log \epsilon|$, we get
\[
\EE\Abs{\hat{C}(X^N)-\hat{C}^E(X^{na})}\leq \mathcal{O}\Paren{\varepsilon C_{\vec{p}}+4+(1+C_{\vec{p}}^{{|\log^{-1} \epsilon|}}+\sqrt{\log C_{\vec{p}}})\sqrt{C_{\vec{p}}}}.
\]
Normalize both sides by $C_{\vec{p}}$. Then,
\[
\EE\Abs{\frac{\hat{C}(X^N)}{C_{\vec{p}}}-\frac{\hat{C}^E(X^{na})}{C_{\vec{p}}}}\leq \mathcal{O}\Paren{C_{\vec{p}}^{|\log^{-1}\epsilon|-\frac12}+\varepsilon}.
\]
\pagebreak
\bibliographystyle{unsrt}
\bibliography{refs}
\end{document}